\documentclass[graybox]{svmult}

\usepackage{mathptmx}       
\usepackage{helvet}         
\usepackage{courier}        
\usepackage{type1cm}        
\usepackage{makeidx}         
\usepackage{graphicx}        
\usepackage{multicol}        
\usepackage[bottom]{footmisc}

\usepackage{amsfonts, amstext,amscd,amstext,latexsym}

\makeindex             

\begin{document}

\title*{On the thin boundary of the fat attractor}
\author{Artur O. Lopes and Elismar R. Oliveira}
\institute{Artur O. Lopes \at Instituto de Matem\'atica-UFRGS, Avenida Bento Gon\c{c}alves 9500 Porto Alegre-RS Brazil, \email{arturoscar.lopes@gmail.com}
\and E. R. Oliveira \at Instituto de Matem\'atica-UFRGS, Avenida Bento Gon\c{c}alves 9500 Porto Alegre-RS Brazil \email{oliveira.elismar@gmail.com}}
%
%
\maketitle

\abstract*{
For, $0<\lambda<1$, consider the transformation $T(x) = d x $ (mod 1) on the circle $S^1$, a $C^1$ function $A:S^1 \to \mathbb{R}$, and, the map  $F(x,s) = ( T(x) , \lambda \, s + A(x))$, $(x,s)\in S^1 \times \mathbb{R}$.  We denote $\mathcal{B}= \mathcal{B}_\lambda$ the upper boundary of the attractor (known as  fat attractor). We are interested in the regularity  of
$\mathcal{B}_\lambda$, and, also in what happens in the limit when $\lambda\to 1$.
We also address the analysis of  following conjecture which were proposed by R. Bam\'on, J. Kiwi,  J. Rivera-Letelier and R. Urz\'ua:
for any fixed $\lambda$, generically $C^1$ on
the potential $A$, the upper boundary $\mathcal{B}_\lambda$ is formed by a finite number of pieces of smooth unstable manifolds of periodic orbits for $F$.
We show the proof of the conjecture for the class of $C^2$ potentials  $A(x)$ satisfying the twist condition (and another technical condition). We do not need the generic hypothesis. When $\lambda$ is  close to $1$ and the potential $A$ is generic a more precise description can be done. We present explicit examples.
In this case the pieces of $C^1$ curves on the boundary have some special properties. Finally, we present the general analysis of the case where $A$ is Lipschitz and its relation with Ergodic Transport.}

\abstract{
For, $0<\lambda<1$, consider the transformation $T(x) = d x $ (mod 1) on the circle $S^1$, a $C^1$ function $A:S^1 \to \mathbb{R}$, and, the map  $F(x,s) = ( T(x) , \lambda \, s + A(x))$, $(x,s)\in S^1 \times \mathbb{R}$.  We denote $\mathcal{B}= \mathcal{B}_\lambda$ the upper boundary of the attractor (known as  fat attractor). We are interested in the regularity  of
$\mathcal{B}_\lambda$, and, also in what happens in the limit when $\lambda\to 1$.
We also address the analysis of  the following conjecture which were proposed by R. Bam\'on, J. Kiwi,  J. Rivera-Letelier and R. Urz\'ua:
for any fixed $\lambda$,  $C^1$ generically on
the potential $A$, the upper boundary $\mathcal{B}_\lambda$ is formed by a finite number of pieces of smooth unstable manifolds of periodic orbits for $F$.
We show the proof of the conjecture for the class of $C^2$ potentials  $A(x)$ satisfying the twist condition (plus a combinatorial condition). We do not need the generic hypothesis for this result. We present explicit examples. On the other hand, when $\lambda$ is  close to $1$ and the potential $A$ is generic a more precise description can be done.
In this case the finite number of pieces of $C^1$ curves on the boundary have some special properties.
Having a finite number of pieces on this boundary is an important issue in a problem related to semi-classical limits and micro-support. This  was consider in a recent  published work by A. Lopes and J. Mohr.
Finally, we present the general analysis of the case where $A$ is Lipschitz and its relation with Ergodic Transport.}


\section{Introduction}

Consider, $0<\lambda<1$, the transformation $T(x) = d \,x $ (mod 1), where $d \in \mathbb{N}$,  a Lipschitz function $A:  S^1 \to \mathbb{R}$, and, the map
\begin{equation}\label{F}
  F(x,s) = ( T(x) , \lambda \, s + A(x)), \;(x,s)\in S^1 \times \mathbb{R}.
\end{equation}

Note that $F^n (x,s) = ( T^n (x) , \lambda^n  \, s + [\lambda^{n-1} A(x)+ \lambda^{n-2} A(T (x))+...+  \lambda A(T^{n-2} (x)) + A( T^{n-1}(x))  ])$. Here  we use sometimes the natural identification of $S^1$ with the interval $[0,1).$
We will assume that $A$ is at least Lipschitz. In this
case results obtained under such hypothesis could  also apply to the shift case (in this setting the potential  $A$ is defined in the Bernoulli space), that is, for $\sigma$ instead of $T$. For certain results in the paper we assume that $A:S^1 \to \mathbb{R}$ is of class $C^1$ or sometimes $C^2$.

The structure of the paper is the following: some results are of general nature and related just to the concept of $\lambda$-calibrated subaction (to be defined later). In this case you need just to assume that $A$ is Lipschitz (see section \ref{ET}).

Other results are related to the dynamics of $F$ and to the conjecture presented in \cite{BKRU}. In this case we assume that the potential $A$ satisfies some differentiability assumptions and also the twist condition (to be defined later). The analysis of the regularity of the boundary of the attractor requires the understanding of $\lambda$-calibrated subactions.
The twist condition assures some kind of transversality condition as we will see.

Note that for $x$ fixed the transformation $F(x,.)$ is bijective over the fiber over $T(x).$
As an illustration we point out that in the case $d=2$,
given $(x,z), $ with $x\in S^1$ and $z \in \mathbb{R}$, consider $x_1$ and $x_2$ the two preimages of $x$ by $T$. Then,
$F\left(x_1, \frac{z- A(x_1)}{\lambda} \right)= (x,z)= F\left(x_2, \frac{z- A(x_2)}{\lambda} \right).$

It is known that
the non-wandering set $\Omega_\lambda$ of $F=F_\lambda$ is a global attractor of the dynamics of $F$: the forward orbit of every
point in $S^1 \times \mathbb{R}$ converges to $\Omega_\lambda$ and $F$ is transitive on $\Omega_\lambda$. In fact, $F$ is topologically semi-conjugate to
a solenoidal map on $\Omega_\lambda$ (see Section 2 in \cite{BKRU}).

In Figure 1 we show the points of the attractor in the case of $T(x)= 2 x$ (mod 1), $\lambda=0.51$ and $A(x) = \sin (2 \, \pi x)$ (see page 1013 in \cite{T1}). In this case the boundary of the attractor is a finite union of smooth curves.

According to \cite{BKRU} $\Omega_\lambda $ is the set of all $(x,s) $ with a bounded infinite backward orbit (i.e., there exists $C >0$ and
$(x_n,s_n),$ $n \in \mathbb{N},$  such that, $F^n (x_n,s_n)=(x,s)$ and
$|s_n|<C$, for  all, $n \in \mathbb{N}$).

This transformation $F$ is not bijective. Anyway, the fiber over $x$ goes in the fiber over $T(x).$ If $s_1<s_2$ is such that
$(x,s_1)$ and $(x,s_2)$ are in the attractor, then $(x,s)$ is in the attractor for any $s_1<s<s_2$ (see section 2.2 in \cite{BKRU}). Note that the iteration of $F$ preserves order on the fiber, that is, given $x$, if $t>s$, then $\lambda s + A(x)< \lambda t + A(x).$

We denote $\mathcal{B}= \mathcal{B}_\lambda$ the upper boundary of the attractor. We are interested in the regularity of $\mathcal{B}_\lambda$ and also in what happens with this boundary in the limit when $\lambda\to 1$. The upper boundary is invariant by the action of $F$.
The analysis of the lower boundary is similar to the case of the upper boundary and will not be consider here.

We show in a rigorous form explicit examples where this boundary is the union of a finite number of $C^\infty$ curves where the tangent angles are never zero (see section \ref{morep}). We also present numerical simulations showing pictures of the boundary in several different cases.

The study of the dimension of the boundary of strange attractors is a topic of great relevance in non-linear physics \cite{Gr} and \cite{Nu}. The papers \cite{HG}, \cite{AY} and \cite{IPRX} discuss somehow related questions. We want to analyze a case where this boundary may not be a union of  piecewise smooth curves.

In Figure~\ref{figquadstandarditera} we show de
image of the  fat attractor for the case of $ A(x) = -(x-0.5)^2$, $\lambda=0.51,$ and $d=2$. In this case the boundary is the union of two  piecewise smooth curves as we will see.

We present in the end of the paper several pictures obtained from computer simulations which illustrate the mathematical results that  we present here (see section \ref{WE}).

We believe that the terminology fat attractor used by M. Tsujii is due to the fact that when $d=2$, $0.5<\lambda \leq 0.51$, and $A(x)=\sin (2 \pi x)$, then, $F$ is such that there exist a SBR probability which is absolutely continuous with respect to the Lebesgue measure on $S^1 \times \mathbb{R}$ (see Example  1 and Figure 1 in \cite{T1}).

It is known that $\mathcal{B}_\lambda $ is the graph of a Lipschitz function $b_\lambda:S^1 \to \mathbb{R}$ (see \cite{T1} and \cite{BKRU}) if $A$ is Lipschitz. We will give a proof of this fact later.
$b_\lambda$ will be called the $\lambda$-calibrated subaction.
\begin{figure}[h!]
  \centering
  \includegraphics[scale=0.5]{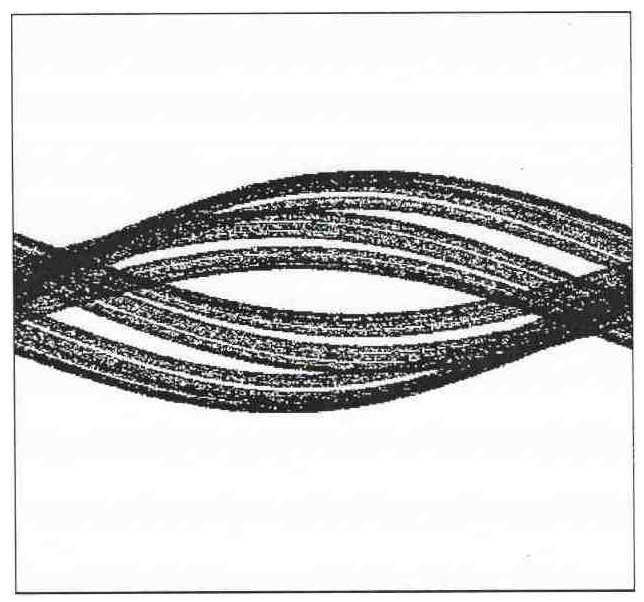}\\
  \caption{The  fat attractor for the case of $ A(x) = \sin(2 \pi x)$, $\lambda=0.51,$ and $d=2$. The picture indicates that the upper boundary is piecewise smooth. It is the envelope of several smooth but  non periodic curves.}\label{figsiniteratsujii}
\end{figure}

One of our main motivations for the present work is the following conjecture (see \cite{BKRU}): for any fixed $\lambda$, generically $C^1$ on
the potential $A$, the upper  boundary $\mathcal{B}_\lambda$ is formed by a finite number of pieces of unstable manifolds of periodic orbits of $F$.

Recently the paper \cite{LMo} shows the importance of having a finite number of pieces on this boundary in a problem related to semi-classical limits and micro-support.

We want to also analyze  cases where this boundary may not be a union of  piecewise smooth curves.

We do not need here the generic hypothesis but we will require the twist condition to be defined later. However, for  a generic potential $A$ more things can be said.

The twist hypothesis is natural in problems on  Optimization (see \cite{Ba}) and in problems on Game Theory (see \cite{Souza}).
The twist property for a potential $A$ is presented in Definition \ref{tui} in Section \ref{OT}.

In the same spirit of \cite{LOS} the idea here is to use an auxiliary family of functions $W_w(x)=W(x,w)$ indexed by $w\in \{1,2,..,d\}^\mathbb{N}$ such that for each $w$ we have  $W_w: (0,1) \to \mathbb{R}$ is, at least $C^1$ (it is $C^2$ in the case we consider). This function  $W$ of the variable $(x,w)$ is called involution kernel. $ W_w$ is not necessarily periodic on $S^1$ (see pictures on section \ref{WE} where a certain $S$ replaces the above $W$).
A natural strategy would be to assume that $A$ satisfies a twist condition and to show that there exists a finite number of points $w_j$, $j=1,2...,k$, and a corresponding set of real values $\alpha_1,\alpha_2,...,\alpha_k$, such that, for each $x\in S^1$ we have
\begin{equation} \label{impo} b_\lambda(x) = \max_{j=1,2,..,k} \{ \alpha_j + W(x,w_j)\},
\end{equation}
where the graph of $b_\lambda$ is $\mathcal{B}_\lambda $.

In \cite{LOT} the results assume, among other things,  that an special point (the turning point) was  eventually periodic. Here we will just use the fact that $A$ satisfies a twist condition.
We will show that the conjecture is  true when  $A$ satisfies a twist condition (see Corollary \ref{twist finite optm} and comments after Corollary \ref{twist finite optm}  on section \ref{oop}). We point out that the twist condition is an open property in the $C^2$ topology. The main problem we analyze here could also be expressed in the $C^2$ topology.

Expressions of the kind (\ref{impo}) appear  in Ergodic Transport (see \cite{LOT} \cite{LOS} \cite{LM} \cite{LMMS2}).   Equation (\ref{bb}) just after Theorem \ref{mm} describes  relation (\ref{impo}) under certain general  hypothesis:  the Lipschitz setting (see section \ref{ET}).

 We apply all the previous results  to the case when $A$ is quadratic in section \ref{qua}.
The main problem we analyze here could also be expressed in the $C^2$ topology.

In section \ref{ET} we describe some general  properties related to Ergodic Transport for the setting we consider here.

\section{$\lambda$-calibrated subactions and $\lambda$-maximizing probabilities} \label{sec2}
\begin{definition}
Given a continuous function $A:S^1\to \mathbb{R}$ and $\lambda\in (0,1)$, we say that a continuous function
{\bf $b_\lambda$ is a $\lambda$-calibrated subaction for $A$}, if for all $x\in S^1$, $ b(x) = \displaystyle\max_{T(y)=x} \{ \lambda \, b(y) + A(y)\}.$
\end{definition}
\begin{figure}[h!]
  \centering
  \includegraphics[scale=0.4,angle=0]{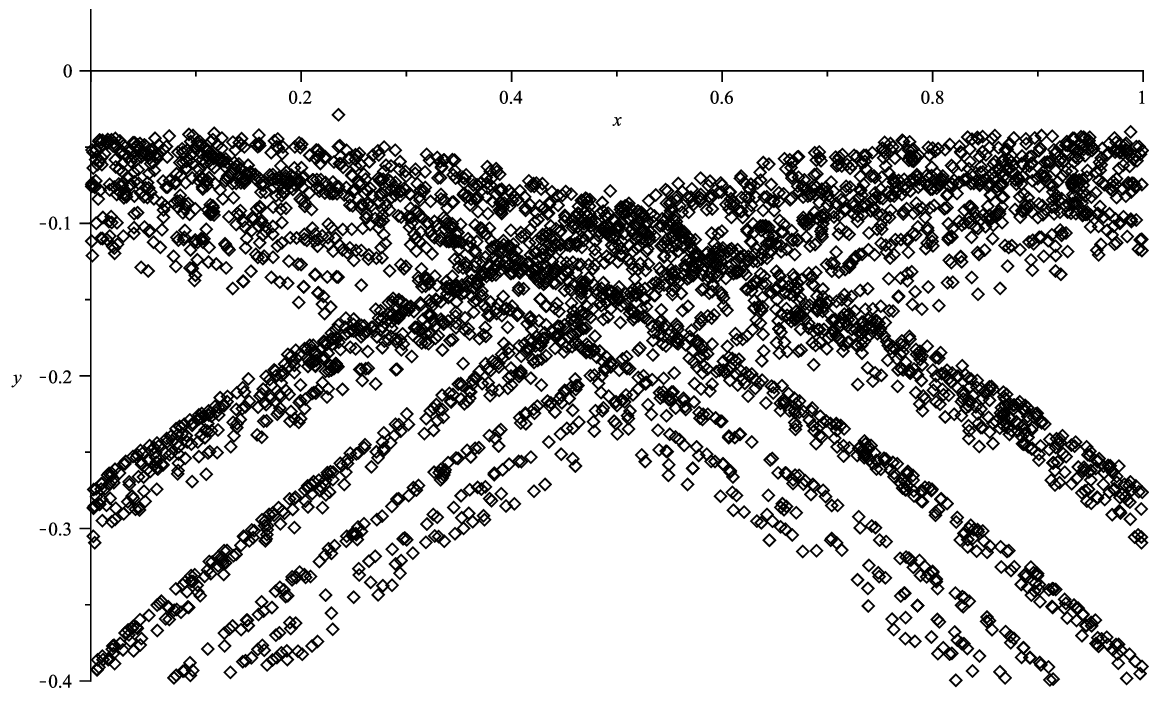}\\
  \caption{The  fat attractor for the case of $ A(x) = -\,(x-0.5)^2$, $\lambda=0.51,$ and $d=2$. The picture indicates that the upper boundary is piecewise smooth. It is the envelop of two smooth but  non periodic curves.}\label{figquadstandarditera}
\end{figure}
A similar concept can be consider when the dynamics is defined by the shift and not $T$ (see section \ref{ET}).

When $A$ is Lipschitz for each $\lambda\in (0,1)$ the function $b_\lambda$ above exist, is Lipschitz and it is unique (see \cite{Bousch-walters} and \cite{LMMS1}).
The  existence of such $b_\lambda$ when $A$ is Lipschitz is also presented in the survey paper \cite{TFM}.

About the interest in such family $b_\lambda$ we can say that
in Aubry-Mather theory and also in Optimization a similar kind of problem is considered in problems related to the so called
infinite horizon discounted Hamilton-Jacobi equation. It provides an alternative method for showing the existence of viscosity solutions (see \cite{It}, \cite{Go1} and \cite{G0}).
Thanks to the formal association with Optimization and Economics the analysis of such family $b_\lambda$, which takes in account values $\lambda\in(0,1)$, can be  called the discounted problem for the potential $A$. If $A$ is Lipschitz it is known (see \cite{BKRU}) that the upper boundary of the attractor is the graph of the Lipschitz $\lambda$-calibrated subaction  $b_\lambda:S^1 \to \mathbb{R}$.

The above result means that if the point $(x,b_\lambda (x))$ is in the upper boundary of the attractor, then, there is a point $y$ such that $T(y)=x$, and $F(y,b_\lambda(y))= (x,b_\lambda (x))$. In this way the analysis of the dynamics of $F$ on the boundary of the attractor is quite related to the understanding of $\lambda$-calibrated subactions.

Note that
if $b$ is  the $\lambda$-calibrated subaction for $A$, then, $ b + \frac{g}{\lambda}$ is the
$\lambda$-calibrated subaction for $A+ \frac{g \circ T}{\lambda} - g .$
In order to obtain our main result on the boundary of the attractor we have to investigate properties of $\lambda$-calibrated subactions.
The three keys elements on our reasoning are: probabilities with support in periodic orbits (see section \ref{sec2}),  a relation of the kind (\ref{impo}) for the function $b$ whose  graph is  the boundary of the attractor (see section \ref{OT}) and
the twist condition (see section \ref{oop}).

We will present now some general results on $\lambda$-calibrated subactions.
We denote by
$\tau_i$, $i=1,2,...,d$, the inverse branches of $T$. For each $i$ the transformation $\tau_i$ has domain $[i-1/d,i/d].$
Given $x$, if $i_0$ is such that
$b(x) = \lambda b( \tau_{i_0} (x))+  A( \tau_{i_0} (x))$, we say  $\tau_{i_0} (x)$ realizes
$b(x)$ (or, realizes $x$). We can also say that $i_0$ is a  symbol which realizes $b(x)$.
One can show that for $d=2$, for any Holder $A$,  there exist $x$ such that $b(x)$ has  two different $\tau_{i_0} (x)$ realizers. In this way realizers are not always unique.
For fixed $x_0\in S^1$ consider $x_1$ such that $b(x_0)= \lambda \, b(x_1) + A(x_1)$, and $T(x_1)=x_0$.
Then, there exist a realizer $i_0$ such that
$\tau_{i_0} (x_0)=x_1$.
Now take $x_2$ such that $b(x_1)= \lambda \, b(x_2) + A(x_2)$  and $T(x_2)=x_1$.
In the same way as before, there exist $i_1$ such that
$\tau_{i_1} (x_1)=x_2$.
In this way get by induction a sequence $x_k\in S^1$ such that $T(x_k) = x_{k-1}.$
This also defines an element $ a=a( x_0)= (i_0,i_1,...,i_n,..) \in \Sigma=\{1,...,d\}^\mathbb{N}$, where
$\tau_{i_k} (x_k)=x_{k+1}$. This $a$  depends of the choice of realizers we choose in the sequence of preimages. We say $(x_0, a(x_0))\in S^1 \times \{1,2,...,d\}^\mathbb{N}$ is an optimal pair. Note that for each $x_0\in S^1$ there exist at least one optimal pair. For each $x_0$ we consider a fixed choice $a(x_0)$, and, the corresponding sequence  $x_k\in S^1$, $k \in \mathbb{N}$.

Consider the probability $\mu_n= \sum_{j=0}^{n-1}\, \frac{1}{n} \,\delta_{x_n}$ and
$\mu_\lambda$ any weak limit  of a convergent subsequence $\mu_{n_k}$, $k \to \infty$.
The probability $\mu_\lambda$ on $S^1$ is $T$ invariant and satisfies
$$\int (b(T(x)) - \lambda\, b(x) -\,A(x))\, d \mu_\lambda=0.$$

Note that $b(T(z)) - \lambda b(z) - A(z)\,\geq 0$
for all $z \in S^1$.
In this way {\bf for $z$ in the support of $\mu_\lambda$} we get the {\bf $\lambda$-cohomological} equation
\begin{equation}  \label{tuto}
b(T(z)) - \lambda b(z) - A(z)\,= 0.
\end{equation}
Therefore, $\mu_\lambda$ is maximizing for the potential $A(z) - b(T(z)) + \lambda b(z).$
For $z$ in the support of $\mu_\lambda$ we have that $F(z, b(z))= (T(z), b(T(z)).$
Moreover, in this case
\begin{equation} \label{tt}
b(T(z)) = \max_{T(y)=T(z)} \{ \lambda \, b(y) + A(y)\}= \lambda \, b(z) + A(z).
\end{equation}
\begin{definition} Any probability which maximizes $A(z) - b(T(z)) + \lambda b(z)$ among $T$-invariant probabilities, where $b$ is the $\lambda$-calibrated subaction, will be called a $\lambda$-maximizing probability for $A$.
\end{definition}

Any $\mu_\lambda$ obtained from a point $x_0$ and a family of realizers described by a certain $a=a(x_0)$ as above is a $\lambda$-maximizing probability for $A$.
Note that $\mu_\lambda$ is not maximizing for $A$ but for the potential $A(z) - b(T(z)) + \lambda b(z)$. General references for maximizing probabilities are \cite{CLT}, \cite{Bousch0},   \cite{GL1}, \cite{J0}  and \cite{TFM}.
As we are maximizing among invariant probabilities we can also say that $\mu_\lambda$ is maximizing for the potential
$$A(z) + (\lambda-1) b(z)=  A(z) - b(T(z)) + \lambda b(z) + [b(T(z)) - b(z) ].$$

\begin{proposition} \label{puxa}  If $z$ is a point in a periodic orbit of period $k$ and moreover $z$ is in the support of the maximizing probability $\mu_\lambda$, then
the realizer $a$ can be taken as a periodic orbit of period $k$ for the shift $\sigma$ acting in the Bernoulli space. We call such probability invariant for the shift of  dual periodic probability.
\end{proposition}
\begin{proof} In order to simplify the reasoning suppose $k=2$.
Note that $T(z)$ is also in the support of the maximizing probability $\mu_\lambda$.
In this case, from (\ref{tuto}) we have that $b(\,T(\, T(z)\,)\,) = \lambda \, b(T(z)) + A(T(z))$
because $T(z)$ is in the support of $\mu_\lambda$.
From equation (\ref{tt})
$$b(z)=b(\,T(\, T(z)\,)\,) = \max_{T(y)=T^2 (z)=z} \{ \lambda \, b(y) + A(y)\}= \lambda \, b(T(z)) + A(T(z)).$$
Therefore,
$\displaystyle  \max_{T(y)=z} \{ \lambda \, b(y) + A(y)\}  = \lambda \, b(T(z)) + A(T(z)).$
In this way we can take the corresponding inverse branch, say $a_1$, and then, say $a_2$,  and we repeat all the way this
choice again and again in order to define $a=(a_1,a_2,a_2,a_4,...)= (a_1,a_2,a_1,a_2,...)$.
In this case $a$ is an orbit of period  two for $\sigma$ and the claim is true.
In the case
$k=3$, note that if $T^3(z)=z$, we have that $b(\,T^2(\, T(z)\,)\,) = \lambda \, b(T^2(z)) + A(T^2(z))$
and $b(\,T(\, T(z)\,)\,) = \lambda \, b(T(z)) + A(T(z))$, because $T(z)$ and $T^2(z)$ are in the support of $\mu_\lambda$. In this way we follow a similar reasoning as before and we get an $a$ which is an orbit of period $3$ for the shift.
Same thing for a periodic orbit with a  general $k$.
\end{proof}

As an example of the above,
suppose $k=2$, then there are two periodic orbits of period 3. Take one of them, let us say $\{z_1,z_2,z_3\}.$ Suppose
$T(z_1)=z_2, T(z_2)=z_3$ and $T(z_3)=z_1$.
Given $z_1$ there exists $a_1$ such that $\tau_{a_1}(z_1)=z_3$. Given $z_3$ there exists $a_2$ such that $\tau_{a_2}(z_3)=z_2$.
Finally, given $z_2$ there exists $a_3$ such that $\tau_{a_3}(z_2)=z_1$. Then, in this case, $a=(a_1,a_2,a_3,a_1,a_2,a_3,a_1,..)$ is in the support of the dual periodic probability for $\{z_1,z_2,z_3\}.$ The set $\{a,\sigma(a), \sigma^2(a)\}$ defines a periodic orbit of period $3$ for the shift $\sigma$.

\begin{definition}
We denote by $R$ the function $R=A - b \circ T + \lambda b\geq 0$ and call it the rate function.
\end{definition}
For fixed $\lambda$, by the fiber contraction theorem \cite{Ro} (section 5.12 page 202 and section 11.1 page 433) we get that the $\lambda$-calibrated subaction  $b=b_{\lambda, A}= b_{A}$ is a continuous function of $A$ in the $C^0$ topology. Moreover, the function $b=b_{\lambda, A}$ is a continuous function of $A$ and $\lambda$.
\medskip

Taking $\lambda \to 1$ will see now that we will get results which are useful for classical Ergodic Optimization.

We denote
$m(A)= \sup\{ \int A \, d \,\nu,$ among $T$-invariant probabilities$\, \nu\,\}$.

We call maximizing probability for $A$ any $\rho$ which attains the supremum $m(A)$. We denote any of these $\rho$ by $\mu_A$.
A continuous function $U: S^1 \to \mathbb{R} $ is called a
calibrated subaction for $A:S^1\to \mathbb{R}$, if, for any $y\in S^1$, we have
\begin{equation}\label{c} U(y)=\max_{T(x)=y} [A(x)+ U(x)-m(A)].
\end{equation}
If $b^\&_\lambda = b_\lambda - \max \, b_\lambda$, then, of course, $\mu_\lambda$ is maximizing for the  rate function  potential $A(x) - b^\&_\lambda(T(z)) + \lambda b^\&_\lambda(z).$
It is known (see  \cite{Bousch-walters} \cite{TFM} \cite{BCLMS} and \cite{LMMS1}) that $b_\lambda$ is equicontinuous in $\lambda$, and, any convergent subsequence, $\lambda_n\to 1$,  satisfies $b^\&_{\lambda_n}\to U$, where $U$ is a calibrated subaction for $A$.

Assuming the maximizing probability for $A$ {\bf is unique} (a generic property according to \cite{CLT}), it is known (see  \cite{GL1} section 4), that when $\lambda\to 1$,  we get that $b^\&_\lambda \to U$, where $U$ is a (the) calibrated subaction for $A$. In  this way we can say that the family $b^\&_\lambda \to U$ selects the calibrated subaction $U$ via the discounted method.
Assuming that the maximizing probability $\mu_A$  is unique, when $\lambda \to 1$, we get that $\mathcal{B}_\lambda$ (after the subtraction of $\max b_\lambda$)  converges to the graph of the calibrated subaction for $A$ in the $C^0$-topology.

 Even if the maximizing probability for $A$  is not unique there exist anyway a unique special limit subaction when $\lambda \to 1$ (see \cite{ILM}). That is, there exist a selection on the discounted method for any Holder potential $A$ (the potential do not have to be generic). In other words, given the potential $A$,  the limit of any sequence $b^\&_{\lambda_n}$,  $n \to \infty$, $\lambda_n \to 1$,  will be
a unique special subaction $U$ for $A$ (independent of the sequence).
A similar property, of course, ia also true for the boundary $\mathcal{B}_\lambda$, when $\lambda \to 1$.

Note that in any case  it is true the relation: for any $z$
$$ b^\&(T(z)) - \lambda b^\&(z) + ( 1- \lambda) (\max b_\lambda)- A(z)\,\geq 0.$$

We point out that in classical Ergodic Optimization, given a Lipschitz potential $A:S^1 \to \mathbb{R}$,  in order to obtain  examples where one can determine explicitly the maximizing probability or a calibrated subaction, it is necessary to know the exact value the maximal value $m(A)$ (see equation (\ref{c})). In the general case this is not an easy task and therefore any method of approximation of this maximal value or associated subaction is helpful. The discounted method provides approximations $b_\lambda$, $\lambda \in (0,1)$, in the $C^0$-topology, of calibrated subactions for $A$ via the Banach fixed point theorem, that is, via a contraction in the set of continuous functions in the $C^0$-topology. You take any function, iterate several times by the contraction and you will get a function $\hat{b}_\lambda$ which is very close to the $\lambda$-calibrated subaction. If $\lambda$ is close to $1$, then, the corresponding $\hat{b}_\lambda^{\&}$ is close to a classical calibrated subaction $U$. In all this is not necessary to know the value $m(A)$. However, when $\lambda$ becomes close to the value $1$ this contraction becomes weaker an weaker.

Theorem \ref{main} claims that for a generic potential $A$ the maximizing measure for $A$ is attained by a $\lambda$-maximizing probability for $\lambda$ in an interval of the form
$[1-\epsilon,1]$.   Thanks to all that one can apply our reasoning for a $\lambda$ is fixed and close to $1$. In the discounted  method taking $\lambda\sim 1$ the procedure also provides a way to approximate the value $m(A)$ as we will see later.

It is known (see \cite{TFM}) that $( 1- \lambda) (\max b_\lambda)\to m(A)$. Then, as we said before, one can get an approximate value of $m(A)$ by the discounted method.

It is easy to see that close by the periodic points the graph of $b$ is a piece of unstable manifold for $F$ (see figure \ref{figstablemanif}).
We point out that for  a generic Lipschitz potential $A$ the maximizing probability for $A$ is a unique periodic orbit (see \cite{CO}).

Now we state a result which is new in the literature.
\begin{figure}[h!]
  \centering
  \includegraphics[scale=0.4,angle=0]{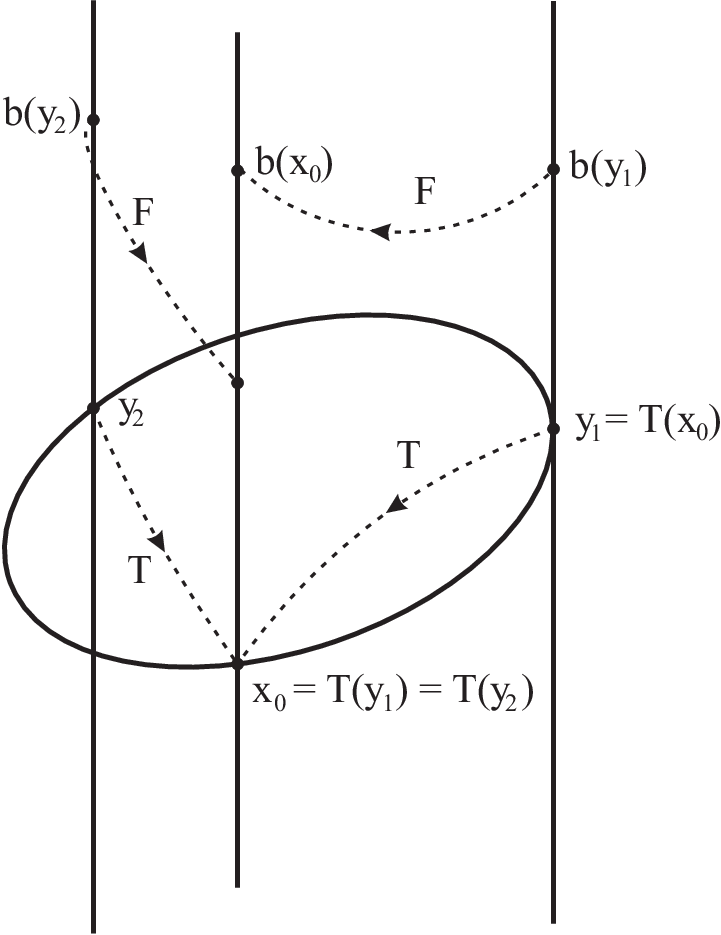}\\
  \caption{A geometric picture of the $\lambda$-calibrated property of $b$. The graph of $b$ describes the upper boundary of the attractor.}\label{figcalibuppbound}
\end{figure}
\begin{figure}[h!]
  \centering
  \includegraphics[scale=0.4,angle=0]{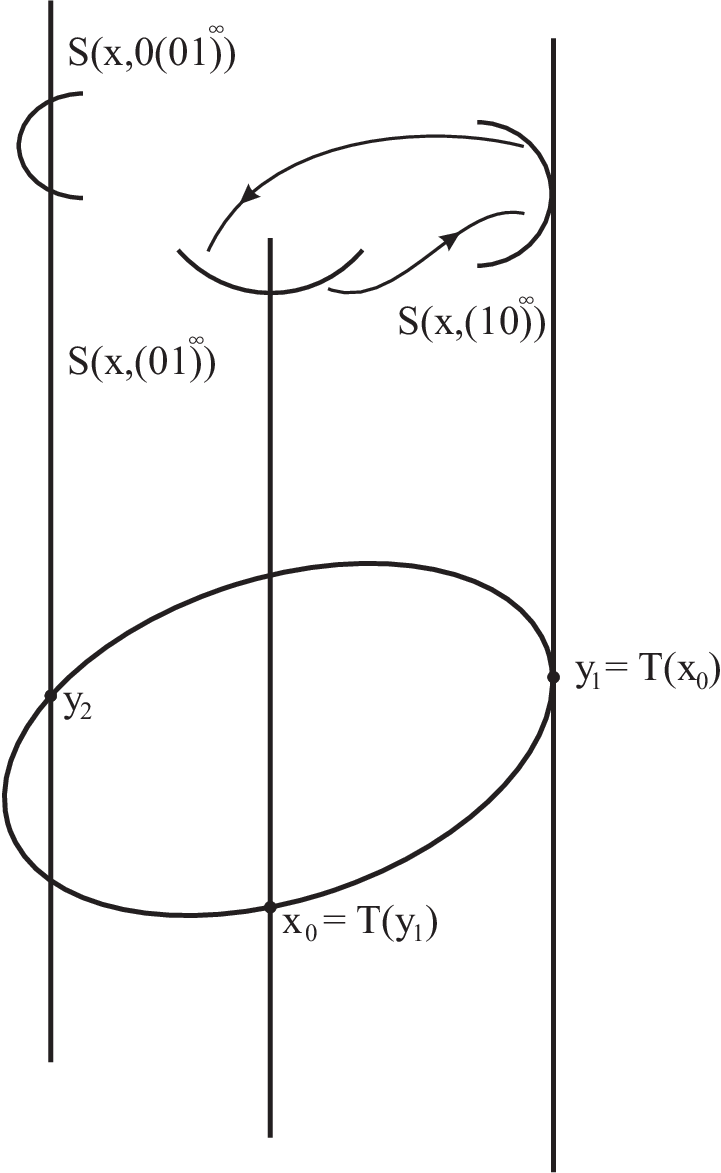}\\
  \caption{The unstable manifolds of a point of period two for $F$.}\label{figstablemanif}
\end{figure}
\begin{theorem} \label{main} If $\nu$ is a weak limit of a converging subsequence  $\mu_{\lambda_n}\to \nu$, $\lambda_n\to 1$, then, $\nu$ is a maximizing probability for $A$. For a generic Lipschitz potential $A$ there exist an $\epsilon$, such that for all $\lambda\in (1-\epsilon,1],$ the  $\lambda$-maximizing probability has support in the  periodic orbit which defines the maximizing probability for $A$. If the potential $A$ is of class $C^1$ the same is true on the $C^1$ topology.
\end{theorem}
\begin{proof}
Consider a subsequence $\mu_{\lambda_n} \to \nu$,  $\lambda_n\to 1$. Such $\nu$ is clearly invariant. Suppose by contradiction that for some $\epsilon>0$ there exists an invariant $\mu$ such that
$ \int (A - U \circ T + U)\, d \nu +\epsilon<
\int ( A - U \circ T + U) d \mu,$
then, for any $n$ large enough we have $\displaystyle \int (A - b_{\lambda_n}  \circ T + \lambda_n b_{\lambda_n} )\, d \mu_{\lambda_n}+\epsilon/2 <
\int ( A - b_{\lambda_n} \circ T + \lambda_n  b_{\lambda_n})d \mu,$
and, we reach a contradiction.

Now, for a generic potential it is known that the maximizing probability for $A$ is a unique periodic orbit (see \cite{CO}). Therefore,
$\mu_{\lambda} \to \nu$, when  $\lambda\to 1$.
From the continuous varying support (see \cite{CLT}) if $\mu_{\lambda} \to \nu$ and $\nu$ is periodic orbit, then, there exist an $\epsilon>0$ such that for $\lambda\in(1-\epsilon,1]$ the probability $\mu_\lambda=\nu$.
If the potential $A$ is of class $C^1$ one can do the following: since $\mu_\lambda$ is maximizing for  $A+ (1-\lambda b)$, which is Lipschitz-close to $A$, then, when $\lambda$ is close  to the value $1$, we apply the continuous varying support property in order to  get   a Lipschitz subaction and perturb in the same way as in \cite{CO}  in order to get an approximation  by another Lipschitz potential with support in a unique periodic orbit. This potential can be approximated once more in the $C^1$ topology and again in his way by the continuous varying support we get $C^1$ potential with support in a periodic orbit. Then, the same formalism as above can be applied.
\end{proof}
If $\mu_\lambda\to \mu_A$, when $\lambda\to 1$, we say that $\mu_\lambda$ selects the maximizing probability $\mu_A$. In the present case for a generic $A$ there is a selection via the discounted method.
\begin{remark}
We point out the final conclusion: for a generic
$A$ we have that for $\lambda$ close to $1$ the maximizing probability  $\mu_\lambda$ is  a periodic orbit. Moreover,  by Proposition~\ref{puxa}
the realizer $a$ for a point $x$ in the support of $\mu_\lambda$ can be taken as a periodic orbit (with the same period) for the shift $\sigma$.
\end{remark}
An interesting example is $A(x)= - (x-0.5)^2$ and $T(x)= 2\, x $ (mod 1), which has a unique
maximizing probability $\mu_A$  which is the one with support in the periodic orbit of period $2$ according to Corollary 1.11 in \cite{J1} (see  also\cite{J2}). Therefore, the
corresponding $\lambda$-maximizing probability $\mu_\lambda$ converges to this one. In fact, there is an $\epsilon$ such that if $1-\epsilon \,<\,\lambda<1 $, then $\mu_\lambda$ has support in this periodic orbit. This example will be carefully analyzed in the lasts sections of the paper.
\section{The $\lambda$-calibrated subaction as an envelope} \label{OT}
Consider (as M. Tsujii in expression (3) page 1014 \cite{T1}) the function $S: (S^1, \Sigma) \to \mathbb{R}$, where $ \Sigma = \{1,2,..,d\}^\mathbb{N}$, given by
\begin{equation}\label{valuefunction}
  S_{\lambda,A} (x,a)=S(x,a) = \sum_{k=0} \lambda^k A(\tau_{k,a}x \,) ,
\end{equation}
where $(\tau_{a_{k-1}}\circ\,...\, \circ\tau_{a_0})\, (x) = \tau_{k,a}x$ and $a=(a_0,a_1,a_2,...).$
For a fixed $a$ the function $S_{\lambda,A} (.,a)$ is continuous up to the point $0$ (see several computer simulations is section \ref{W} and the explicit expression for the quadratic case in subsection \ref{morep}).
Note that if $A$ is of class $C^2$, then for a fixed $a$ the function $S(. ,a)$ is smooth up to the point $0$ in $S^1$, if $1>\lambda>\frac{1}{d}$ (see in page 1014
  the claim between expressions (3) and (4) in \cite{T1}).

We point out that the upper boundary of the attractor is periodic but each individual $S(x,a)$ as a function of $x$  is not (see worked examples in the end of the paper).
  Note also that for $\lambda$ and $a$  fixed the function  $S_{\lambda, A}(. ,a)$ is linear in $A$.
All $x$ has a corresponding $a=a(x)$ such that $b(x) = S(x,a).$ Indeed, for the given  $x$ take
$i_0$ such that
$ b(x) = \lambda b (\tau_{i_0} (x)) + A( \tau_{i_0} (x)),$
then, take
$i_1$ such that
$ b( \tau_{i_0} (x)) = \lambda b ((\tau_{i_1}\circ \tau_{i_0}) (x)) + A((\tau_{i_1}\circ \tau_{i_0} (x)),$
and so on. In this way we get $a=(i_0,i_1,i_2,..).$ This $a$ is not necessarily unique.
We call any such possible $a(x)$ a {\bf realizer for $x$}.
Note that
$$ b(x) = \lambda\,[\,\lambda u ((\tau_{i_1}\circ \tau_{i_0}) (x)) + A((\tau_{i_1}\circ \tau_{i_0} (x)) \,] + A( \tau_{i_0} (x))=$$
$$ \,\lambda^2 u ((\tau_{i_1}\circ \tau_{i_0}) (x)) + \lambda \,A((\tau_{i_1}\circ \tau_{i_0} (x)) \, + A( \tau_{i_0} (x))=$$
$$ \,\lambda^n u (( \tau_{i_{n-1}}\circ\,...\,\circ\tau_{i_1}\circ \tau_{i_0}) (x)) +$$
$$\lambda^{n-1}  \,A (( \tau_{i_{n-1}}\circ\,...\,\circ\tau_{i_1}\circ \tau_{i_0}) (x))+  ...+ \lambda \,\, A((\tau_{i_1}\circ \tau_{i_0} (x)) \, + A( \tau_{i_0} (x)).$$
Taking the limit when $n\to \infty$ we get $b(x) = S(x,a).$

From the construction we claim  that for any other $c \in \{1,2,..,d\}^\mathbb{N}$ we have
$b(x) \geq S(x,c).$ Indeed, consider
$z(x) = \limsup_{n \in \mathbb{N}}  \{
\lambda^{n-1} \,A (( \tau_{i_{n-1}}\circ\,...\,
\circ\tau_{i_1}\circ \tau_{i_0}) (x))+ ...$
 $...+ \lambda \,\, A((\tau_{i_1}\circ \tau_{i_0} (x)) \, + A( \tau_{i_0} (x))\,|\, (i_0,i_1,...,i_{n-1}) \in \{1,2,..,d\}^n \},$
and, the operator
$\displaystyle\mathcal{L}_\lambda (v)(x) = \sup_{i=1,2...,d} [A(\tau_i(x)) + \lambda v(\tau_i(x))].$
Then, $\mathcal{L}_\lambda  (z)=z.$ It is known that $b$ is a fixed point for $\mathcal{L}_\lambda $ (see section 3 in \cite{TFM}, or section 2 in \cite{BCLMS}). From the uniqueness of the fixed point it follows the claim. Therefore, we get from above the following result which we call the Envelope Theorem.
\begin{theorem}\label{envthm}
$b(x)=\displaystyle b_{\lambda,A}(x)= \sup_{ c \in \{1,2,..,d\}^\mathbb{N}} \, S(x,c)=S(x,a(x)),$
where $a(x)$ is a realizer for $x$.
\end{theorem}
As the supremum of convex functions is convex we get that for $\lambda$ fixed  the function $b_{\lambda,A}$
varies in  convex way with $A$.

Figure~\ref{figquadstandarditera} suggests that $b(x)$ is obtained as $\sup \{ S(x,w_1), S(x,w_2)\}$, where, $w_1,w_2$, is  in  $\Sigma$.
Later we will show that in several interesting examples  $w_1,w_2$ are in a periodic orbit of period $2$ for the shift $\sigma$.
 As we said before in the introduction we point out again here (in a more precise way)  that Figure 1 in \cite{T1} suggests that $b(x)$ is obtained as $\sup \{ S(x,w_i), i=1,2,3,4\}$, where, $w_i$, $i=1,2,3,4$, are in  $\Sigma$. Note that the potential $A$ in that case is conjectured to have a maximizing probability in an orbit of period $4$ (see \cite{CoG}).

Note that if $A$ is of class $C^2$, then, $S_a:(0,1) \to \mathbb{R}$ is of class $C^2$.
We define $\pi(x)=i$, if $x$ is in the image of
$\tau_i(S^1- \{0,1\})$, $i\in \{1,2,...,d\}$.

Note also (see (7) page 1014 in \cite{T1}) that
$S(T(x), \pi(x) a)= A(x) + \lambda\, S(x,a).$
Or, in another way, for any $a=(a_0,a_1,...)$ we have that
$S(x,a)= A(\tau_{a_0}(x)) + \lambda S(\tau_{a_0}(x), \sigma(a)). $
This also means that  $\phi(x,a)=(x,S(x,a))$ is a change of coordinates from $F$ to $\mathbb{T}(x,a)= (T(x), \pi(x) a)$ \cite{T1}.
$\mathbb{T}(x,a)$ is forward invariant in the upper boundary $\mathcal{B}$ of the attractor.
Note also that $\mathbb{T}^{-1} (x,a)= (\tau_{a_0}(x), \sigma(a))$ (when defined).
\begin{definition}
Consider a fixed $\overline{x}\in S^1$  and variable $x\in S^1$, $a\in\{1,2\}^\mathbb{N}$, then we
define
\begin{equation} \label{W}
W(x,a)= S(x,a) - S(\overline{x},a).
\end{equation}
We call such $W$ the $\lambda$-involution kernel for $A$.
\end{definition}
Note  that for $a$ fixed $W(x,a)$ is smooth on $x\in(0,1).$
From the above definition we get,
$$ \lambda W(\tau_{a_0} (x), \sigma(a)) - W(x,a)=$$
$$ [\lambda S(\tau_{a_0} (x), \sigma(a))- \lambda S(\overline{x}, \sigma(a))]-
  [S(x, a)-  S(\overline{x}, a)]=$$
$$ [\lambda S(\tau_{a_0} (x), \sigma(a))- S(x, a)]-
[\lambda S( \overline{x}, \sigma(a))- S(\overline{x}, a)]=$$
$$-  A(\tau_{a_0}(x))+   [\lambda S(\overline{x}, \sigma(a))- S(\overline{x}, a)]. $$
Note that  $[\lambda S(\overline{x}, \sigma(a))- S(\overline{x}, a)]$ just depend on $a$ (not on $x$).
\begin{definition}
If we denote
$A^* (a) = [\lambda S(\overline{x}, \sigma(a))- S(\overline{x}, a)],$
we get the {\bf $\lambda$-coboundary equation}: for any $(x,a)$
$$A^* (a) = A(\tau_{a_0}(x))+ [\lambda W(\tau_{a_0} (x), \sigma(a)) - W(x,a)].$$
We say that $A^*$ is the $\lambda$-dual potential of $A$.
\end{definition}
Note that $A$ is defined for the variable $x\in S^1$ and $A^*$ is defined for $a$ which is in the dual space $\Sigma$.
The above definition  is similar to the one presented in \cite{BLT}, \cite{LOT}, \cite{CLO} and \cite{LMMS1}.
Note that we have an explicit expression for the $\lambda$-involution kernel $W$ which appears in the above definition.

Below we consider the lexicographic order in $\{1,...,d\}^\mathbb{N}$.
\begin{definition} \label{tui} We say that $A$ satisfies the the twist condition, if an (then, any) associated involution kernel $W$, satisfies the property: for any $a<b$, we have
$$\frac{\partial W}{\partial x}(x,a) - \frac{\partial W}{\partial x}(x,b)>0.$$
\end{definition}

Note that this condition does not imply that there is a uniform positive lower bound for $\frac{\partial W}{\partial x}(x,a) - \frac{\partial W}{\partial x}(x,b)$ when $b>a$.

{\bf
It is equivalent to state the above relation for $S$ or for $W$.}
\begin{figure}[h!]
  \centering
  \includegraphics[scale=0.4]{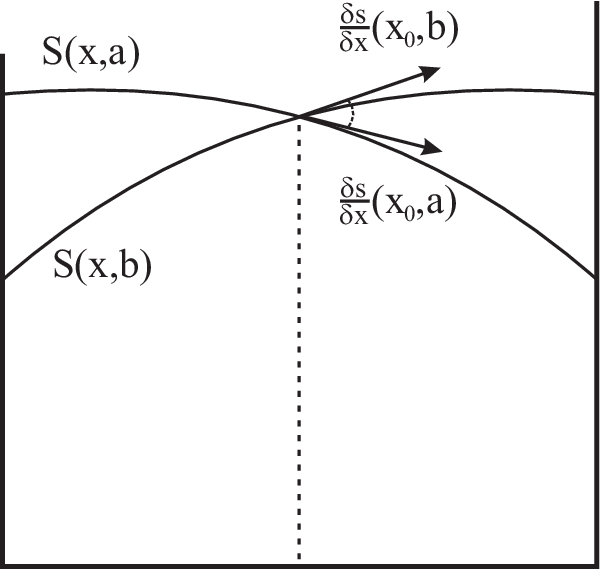}\\
  \caption{Under the twist condition the way the two graphs cut is compatible with the inequality $b<a$ (see Proposition \ref{ooi}).}\label{figtwist}
\end{figure}
An important issue is described by Proposition 2.1 in \cite{CLO} which basically says that (in our context) in the case $W$ satisfies the twist property, then
association $x$ to a realizer $a(x)$ is monotone, where we use the lexicographic order in $\{1,...,d\}^\mathbb{N}$. See also Proposition \ref{titi} in the last section. This is not exactly, but very close, of saying that the
support of the optimal plan probability for $W$ is a graph. In \cite{LOT} the question about the property of cyclically monotonicity (in the support)  is addressed.
\section{Geometry, combinatorics of the graphs of $S(. ,a)$ and the twist condition} \label{oop}
Remember that
$S(x,a) = \sum_{k=0} \lambda^k A( \tau_{k,a}x) ,$
and
$W(x,a)= S(x,a) - S(\overline{x},a).$
Moreover, one can get the calibrated subaction via the \textbf{superior envelope}
$ b(x)= b_{\lambda,A}(x)= \sup_{ a \in \Sigma} \, S(x,a)=S(x,a(x)).$
In this case $a(x) \in \Sigma$ is called an optimal symbolic element for $x$ (possibly not unique).
Under the twist condition, two graphs cut one each other in a   compatible way with the inequality $a<b$.
The envelope result, assures that, if the family $S(x,\cdot)$ is continuous in $\Sigma$, then $\frac{\partial b_{\lambda,A}}{\partial x}(x,a)= \frac{\partial S}{\partial x}(x,a)$, for every optimal $a$. Thus, if $A$ is twist the optimal symbolic element is unique in every differentiable point of $b(x)=b_{\lambda,A}(x)$. {\bf The two graphs on the Figure~\ref{figtwist} can not cut twice by the twist property. This is the purpose of the next results. Note, however, that the graph of one $S(. ,c)$ will be intersected by an infinite number of other graphs of $S(. ,d)$.}

We will study now some additional properties of the family of maps $S(x,a)$. The first step is to consider some especial functions.
\begin{definition}For a fixed pair $a, b \in \Sigma$ we define $\Delta: S^1 \times \Sigma \times \Sigma \to \mathbb{R}$ by
$\Delta(x,a,b)= S(x,a)- S(x,b),$
that is $C^2$ smooth on $x\in(0,1).$
\end{definition}
Computing this derivatives we get
$\Delta'(x,a,b)= S'(x,a)- S'(x,b)$   and $\Delta''(x,a,b)= S''(x,a)- S''(x,b),$
thus we get two consequences. The first: if $A$ is twist and $a \neq b$ then $\Delta'(x,a,b) \neq 0$, more precisely, if $a<b$ then $\Delta'(x,a,b) > 0$ else, if $a>b$ then $\Delta'(x,a,b) < 0$. The second consequence is for quadratic potentials, if $A$ is quadratic then $A''=cte$ and this implies that $ \Delta''(x,a,b)=0$, thus
$\Delta(x,a,b)= \Delta(0,a,b)+ x \Delta'(0,a,b),$
for $x \in S^1$.
The twist property give us a certain geometric structure on the family $S(x,a)$. If we assume $a$ is optimal for $x=0$ and define $a^-=\{w \in \Sigma, \; w<a \}$ and $a^+=\{w \in \Sigma, \; w>a \}$ we get the picture described by Figure~\ref{figoptimaltwistorder}.

Indeed, $\Delta(0,a,b) >0$ and $\Delta'(x,a,b) >0$ because $b>a$, thus $\Delta(x,a,b)$ is increasing what means that $S(x,a)$ and $S(x,b)$ has no intersection. On the other hand $c \in a^-$ which means that $\Delta'(x,a,b) < 0$ thus $S(x,c)$ can intersect $S(x,a)$ in just one point.
\begin{figure}[h!]
  \centering
  \includegraphics[scale=0.7,angle=0]{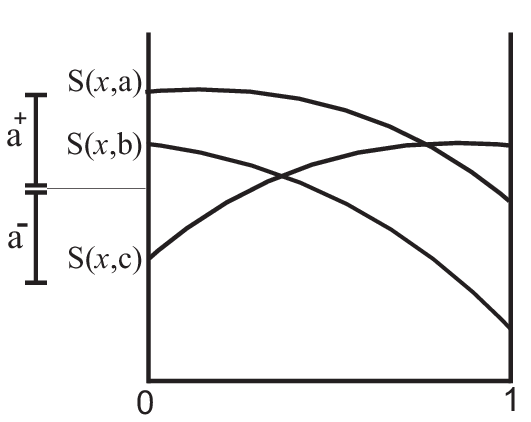}\\
  \caption{$\,b \in a^+$ and $c \in a^-$}\label{figoptimaltwistorder}
\end{figure}
Reciprocally, the twist property allow us to determinate the exact order of every three members $a,b$ and $c$ from the geometrical position of $S(x,a), S(x,b)$ and $S(x,c)$. We call this the triangle property; this  means that if the corresponding positions are as in the Figure~\ref{figtwistorder}, then, we get that   $a>c>b$.
\begin{figure}[h!]
  \centering
  \includegraphics[scale=0.7,angle=0]{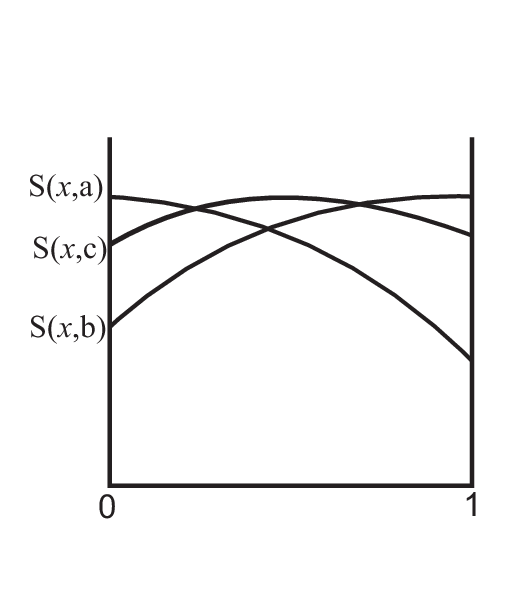}\\
  \caption{Triangle property.}\label{figtwistorder}
\end{figure}
\begin{proposition} \label{ooi}
Suppose $S$ satisfies the twist condition  for some fixed $a$ and $b$, the positions of the graphs of $S(.,a)$ and $S(.,b)$ are described by  figure 5.  We assume $x_0$ is such that
$0=\Delta(x_0,a,b)= S(x_0,a)- S(x_0,b),$
then
$\frac{\partial S}{\partial x} (x_0,a) < \frac{\partial S}{\partial x} (x_0,b).$
\end{proposition}
\begin{proof}
The proof follows from the fact that
$\frac{\partial S}{\partial x} (x_0,a) \,-\, \frac{\partial S}{\partial x} (x_0,b)=\Delta'(x_0,a,b) < 0.$
\end{proof}
The twist property assures a transversality condition on the intersections of the leaves described by the different graphs of $S(.,a)$ (see beginning of section 4 in \cite{T1}). We point out that the twist condition was not explicitly considered in \cite{T1}.
\subsection{Invariance properties of the envelope}
We already know that $ b(x)$  given by $b_{\lambda,A}(x)= \sup_{ a \in \Sigma} \, S(x,a)=S(x,a(x)),$ is the upper envelope of the family $S(x,a)$. We remind the reader that the map $\mathbb{T}^{-1}$ is defined by  $\mathbb{T}^{-1} (x,a)= (\tau_{a_0}(x), \sigma(a))$. It is also  well defined
$$S(x,a)= A(\tau_{a_0}(x)) + \lambda S(\tau_{a_0}(x), \sigma(a)).$$
We will prove that  the upper envelope of the family $S(x,a)$ is invariant by $\mathbb{T}^{-1}$.
Abusing of the notation we set $a(x)$ as \textbf{\emph{the set of such solutions for a given fixed $x$}}; $a(x)$ is indeed a multi function, that is, $b(x)= S(x,a), \; \forall a \in a(x)$. We are going to prove that the first symbol in $a(x)$ uniquely determined by  the symbol $i_{0}$ that turns out to be the maximum $ b(x) = \max_{i} \{ \lambda \, b(\tau_{i} x ) + A(\tau_{i} x)\},$ more precisely, $b(x) = \lambda \, b(\tau_{i_{0}} x ) + A(\tau_{i_{0}} x)$.

We begin with a technical and crucial lemma.
\begin{lemma} \label{angle} If $A$ is twist and $a>b$, with $d(a,b) = \frac{1}{2^{N}}$ then the angle $\alpha$ between two intersecting $S(x,a)=S(x,b)$ satisfy
$$\tan(\alpha)= \Delta'(x,a,b)  \leq \|A'\|_{\infty} \left(\frac{\lambda}{2}\right)^{N} \frac{2}{2- \lambda}.$$
\end{lemma}
\begin{proof} As $A$ is twist $S(x,a)$ and  $S(x,b)$ are transversal and the positive angle is given by $\tan(\alpha)= \Delta'(x,a,b)$.
$$\Delta'(x,a,b) = \frac{\partial S}{\partial x}(x,a) - \frac{\partial S}{\partial x}(x,b) = \sum_{k=0}^{\infty} \lambda^{k} (A'(\tau_{k,a} x)  - A'(\tau_{k,b} x))  \frac{1}{2^{k+1}} $$

But  $$ \frac{1}{2 } \sum_{k=n}^{\infty} \left(\frac{\lambda}{2}\right)^{k} |A'(\tau_{k,a} x)  - A'(\tau_{k,b} x)|
\leq  \|A'\|_{\infty}  \sum_{k=n}^{\infty} \left(\frac{\lambda}{2}\right)^{k}=$$ $$ =  \|A'\|_{\infty} \left(\frac{\lambda}{2}\right)^{N} \frac{1}{1- \frac{\lambda}{2}} =\|A'\|_{\infty} \left(\frac{\lambda}{2}\right)^{N} \frac{2}{2- \lambda}.$$
\end{proof}

\medskip

We point out that we do not need to take $\lambda$ close to $1$ for the above result.
Now, we want to show that for any fixed $\lambda$, under the twist condition plus another technical condition, there exist a finite number of points $c_j$, $j=1,2,..,k$, such that
$$b(x)= \sup_{ c \in \{1,2,..,d\}^\mathbb{N}} \, S(x,c) =
\sup_{j=1,2,..,k} \, S(x,c_j).$$

A natural question is to ask about the nature of these points $c_j$, $j=1,2,..,k$. In the case the $\lambda$-maximizing probability is a unique periodic orbit and $A$ is twist  we will be able to describe some properties (see section \ref{tre}). Some properties depend of the combinatorics of the position of the orbits (see Theorem \ref{okl}). It can happen (see example below) that the $\lambda$-maximizing probability is a periodic orbit of period $2$ and we need to use $3$ points $c_1,c_2,c_3$ in the above equation (see figure \ref{figdriftiter}).

\medskip

\begin{lemma} \label{invariance}
If $ b(x) = \lambda \, b(\tau_{i_{0}} x ) + A(\tau_{i_{0}} x)$ then
$i_{0}*a(\tau_{i_{0}} x)\,\in a(x),$
where $*$ means the concatenation.
\begin{figure}[h!]
  \centering
  \includegraphics[scale=0.6,angle=0]{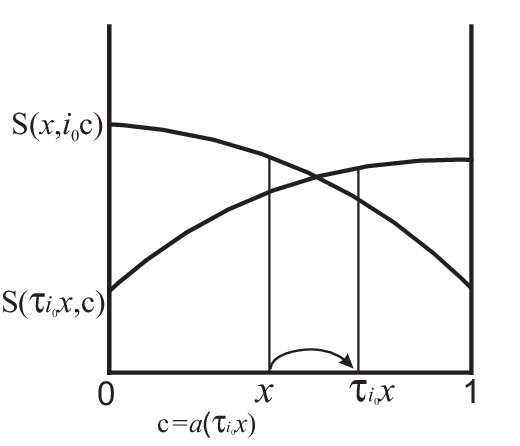}\\
  \caption{ }\label{figinvariance}
\end{figure}
Reciprocally, if $b(x)= S(x,c)$, then, $b(\tau_{c_{0}} x ) = S( \tau_{c_{0}} x, \sigma c)$, and $b(x) = \lambda \, b(\tau_{c_{0}} x ) + A(\tau_{c_{0}} x)$. In other words, if $b(x)= S(x,c)$ then the first symbol $i=c_{0}$ of $c$ attains the supremum $ b(x) = \max_{i} \{ \lambda \, b(\tau_{i} x ) + A(\tau_{i} x)\}$.
\end{lemma}
\begin{proof}
Suppose that  $ b(x) = \lambda \, b(\tau_{i_{0}} x ) + A(\tau_{i_{0}} x)$ and $c \in a(\tau_{i_{0}} x)$,
then $b(\tau_{i_{0}} x )= S( \tau_{i_{0}} x, c)$. An easy computation shows that
$\lambda \, b(\tau_{i_{0}} x ) + A(\tau_{i_{0}} x) = A(\tau_{i_{0}} x)+ \lambda \, S( \tau_{i_{0}} x, c) = S(x, i_{0}*c)$, so $b(x) = S(x, i_{0}*c)$ which means that  $i_{0}*c \in a(x).$

For the reciprocal suppose   $b(x)= S(x,c)= A(\tau_{c_{0}} x)+\lambda S( \tau_{c_{0}} x, \sigma c)$. Since $ b(x) = \max_{i} \{ \lambda \, b(\tau_{i} x ) + A(\tau_{i} x)\}\geq  \lambda \, b(\tau_{c_{0}} x ) + A(\tau_{c_{0}} x),$ we get,
$$A(\tau_{c_{0}} x)+\lambda S( \tau_{c_{0}} x, \sigma c) \geq \lambda \, b(\tau_{c_{0}} x ) + A(\tau_{c_{0}} x),$$ which is equivalent to  $b(\tau_{c_{0}} x ) \leq S( \tau_{c_{0}} x, \sigma c)$, thus $b(\tau_{c_{0}} x ) = S( \tau_{c_{0}} x, \sigma c)$.  Substituting this in the previous equation we have that $b(x)= S(x,c)= A(\tau_{c_{0}} x)+\lambda b(\tau_{c_{0}} x )$.
\end{proof}
If we suppose additionally that $W$ satisfies the twist condition then, if $a(\tau_{i_{0}} x)$ is not a single point, then by Proposition \ref{ooi}, the function $b$ is not differentiable at $x$ because $b(x)= S(x,i_{0}*c)$ and $b(x)= S(x,i_{0}*d)$. However, $S'(x,i_{0}*c) \neq S'(x,i_{0}d)$, if $c\neq d$, where $c,d \in a(\tau_{i_{0}} x)$.
\begin{corollary} \label{spread}  The set $\Omega=\{(x,a) \in [0,1] \times \Sigma \; | \; b(x)= S(x,a)\}$ is $\mathbb{T}^{-1}$-invariant.
\end{corollary}
\begin{proof}
Indeed, if $(x,a) \in \Omega$ then $b(x)= S(x,a)$ and by Lemma~\ref{invariance} we have $b(\tau_{c_{0}} x )= S( \tau_{c_{0}} x, \sigma c)$. Thus $\mathbb{T}^{-1}(x,a) \in \Omega$.
\end{proof}
\begin{definition} A crossing point $x = x_{ab}$ is the single point $x$ satisfying $S(x,a)= S(x,b)$ with $a>b$.
\end{definition}
\begin{figure}[h!]
  \centering
  \includegraphics[scale=0.6,angle=0]{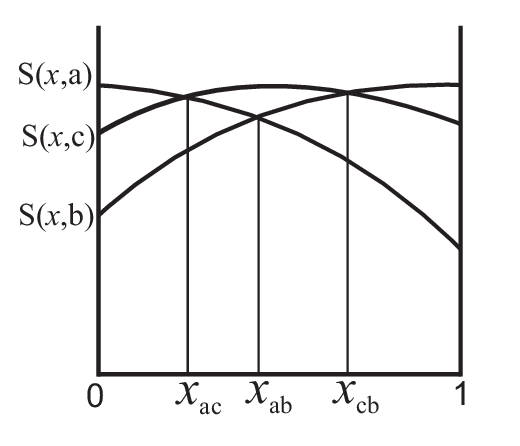}\\
  \caption{ }\label{figcrossing}
\end{figure}
When $A$ is twist, the crossing points are ordered according to the order of $a, b$ and $c$ as in the  above figure.

\begin{definition}
 A turning point is a point  $x$ such that $ b(x) = \lambda \, b(\tau_{i_{0}} x ) + A(\tau_{i_{0}} x)$ for more than one symbol $i_0$. The concept of turning point was introduced in \cite{CLO} and \cite{LOS}. A turning point is \textbf{simple} if its forward orbit is finite.
\end{definition}



\begin{corollary} \label{twist finite optm} Assume $A$  satisfies the twist condition and moreover that  there exists a finite number of turning points and that each one is simple, then the boundary of the attractor is given by a finite number of $C^2$ pieces (of unstable manifolds).
\end{corollary}
\begin{proof} Let $x$ be a point on the boundary of the attractor where the optimal symbolic changes, see Figure~\ref{figdriftiter}.
If $S(x,a)=b(x)=S(x,c)$ and $a_i=c_i$ for $i=0,..., N-1$ we get from Lemma~\ref{invariance} that $b(x)= S(x,c)$, then, $b(\tau_{c_{0}} x ) = S( \tau_{c_{0}} x, \sigma c)$ and $b(x)= S(x,a)$, then, $b(\tau_{a_{0}} x ) = S( \tau_{a_{0}} x, \sigma a)$. Choosing $x_1=\tau_{a_{0}} x$ we get $S( x_1, \sigma a)= S( x_1, \sigma c)$. Proceeding in this way we obtain $S( x_{N-1}, \sigma^{N-1} a)= S( x_{N-1}, \sigma^{N-1} c)$. What means that $z= x_{N-1}$ is a turning point  and $T^{N}(z)=x$. In this way, we conclude that any point $x$ such that $S(x,a)=b(x)=S(x,c)$ lies in the orbit of a turning point. Since the number of turning points is finite and its orbits are finite, because they are simple, we obtain that there is just a finite number of this points. Finally, by the twist property we guarantee that $\sharp \{x \, | \, S(x,a)=b(x)=S(x,c)\}=1$ that is, the number of pieces in the boundary is finite.
\end{proof}

We will present later examples where $b$ is explicit and have a finite number o realizers. Therefore
the boundary of the attractor is given by a finite number of $C^2$ pieces (see section \ref{morep}).

In general,  explicit computations are very difficult to find, but we will present some computational evidence to illustrate  the conclusion of  Corollary~\ref{twist finite optm}.
\begin{example} Take $\varepsilon=0.005$, $\lambda=0.51$, $drift= 0.05$ and $gap=0.001$, we use a truncated version $S(x,a)=  \sum_{k=0}^{7} \lambda^k A(\,\tau_{k,a}x)$ where $A(x):=A_{\varepsilon}(x)= - \left(  1.010\,x- 0.455 \right) ^{2}$ is a perturbation of $- \left(  \,x- 0.5 \right) ^{2}$ by $A_{\varepsilon}(x)=- \left(   x- 0.5  + \varepsilon \phi(x) + drift \right) ^{2}$ and $\phi(x)=2x-1$. The figure below shows the maximum $b(x)=\max_{a}S(x,a)$ in a grid of 25 divisions of $[0,1]$, and suggest the form of the graph of $b$ in the figure \ref{figpotdriftcomp}.
\begin{figure}[h!]
  \centering
  \includegraphics[scale=0.4,angle=0]{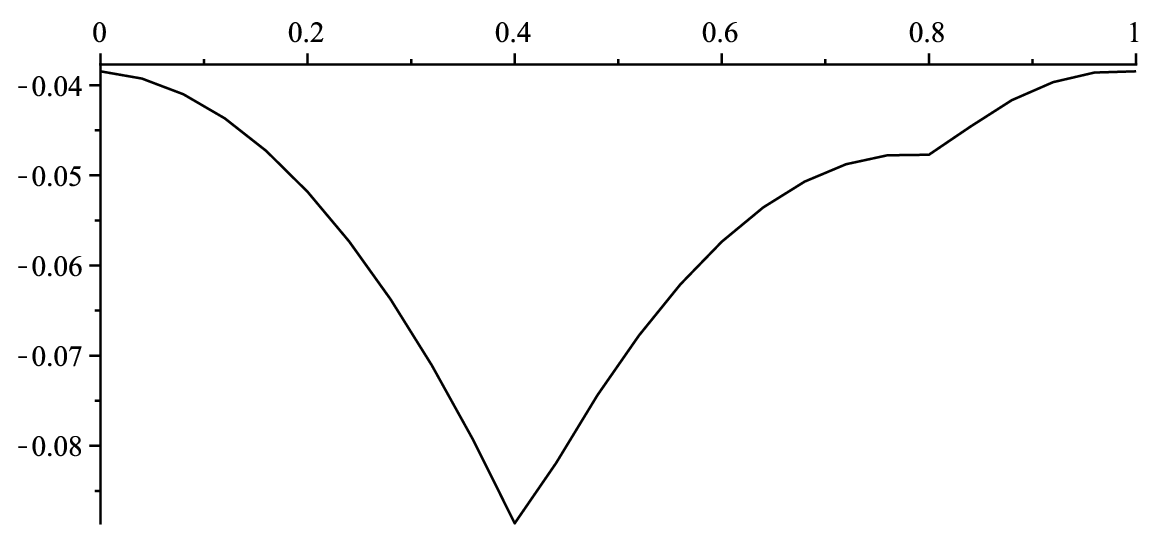}\\
  \caption{ }\label{figpotdriftcomp}
\end{figure}
This figure suggest that there is 3 pieces $S(x,10101....)$, $S(x,01010....)$ and a unknown  $S(x,c_{0}c_{1}c_{2}...)$. Since the perturbed potential still having the  twist property we get $c_0 =0$. Taking $x$ in the right side of de second crossing point $v$ we get $b(x)= S(x,c_{0}*\sigma c)$ and from Corollary~\ref{spread} we get $b(\tau_{c_{0}} x )= S( \tau_{c_{0}} x, \sigma c)$, in particular $b(\tau_{c_{0}} x )=b(\tau_{0} x )$ lies in the right side of the first crossing point $u$  because the first symbol of the optimal sequences  for points before $u$ is 1. Therefore, $\sigma c$ should be $(0101...)$. From this hypothetic deductions we can suppose that, if $u$ and $v$ are the crossing points, then the formula for the superior envelop $b(x)$ should be
$$b(x):=
\left\{
  \begin{array}{lcl}
    S(x,101010....)&, & 0< x \leq u\\
    S(x,010101....)&, & u< x \leq v\\
    S(x,001010....)&, & v< x \leq 1
  \end{array}
\right.
$$
This is exactly what the Figure~\ref{figdriftiter} shows, that is, we already know from Corollary~\ref{invariance} that there is only a finite number of pieces, we just deduce now what are the geometric positions of these pieces. In the graph we plot $b(x)+gap=b(x)+ 0.001$ in order to distinguish the difference between this and the picture computationally obtained.  If we iterate some orbits close to the attractor by the transformation $F(x,s)=(T(x), \lambda s + A(x))$, we can see that there is numerical evidence that our claim is true in this particular case.
\begin{figure}[h!]
  \centering
  \includegraphics[scale=0.6,angle=0]{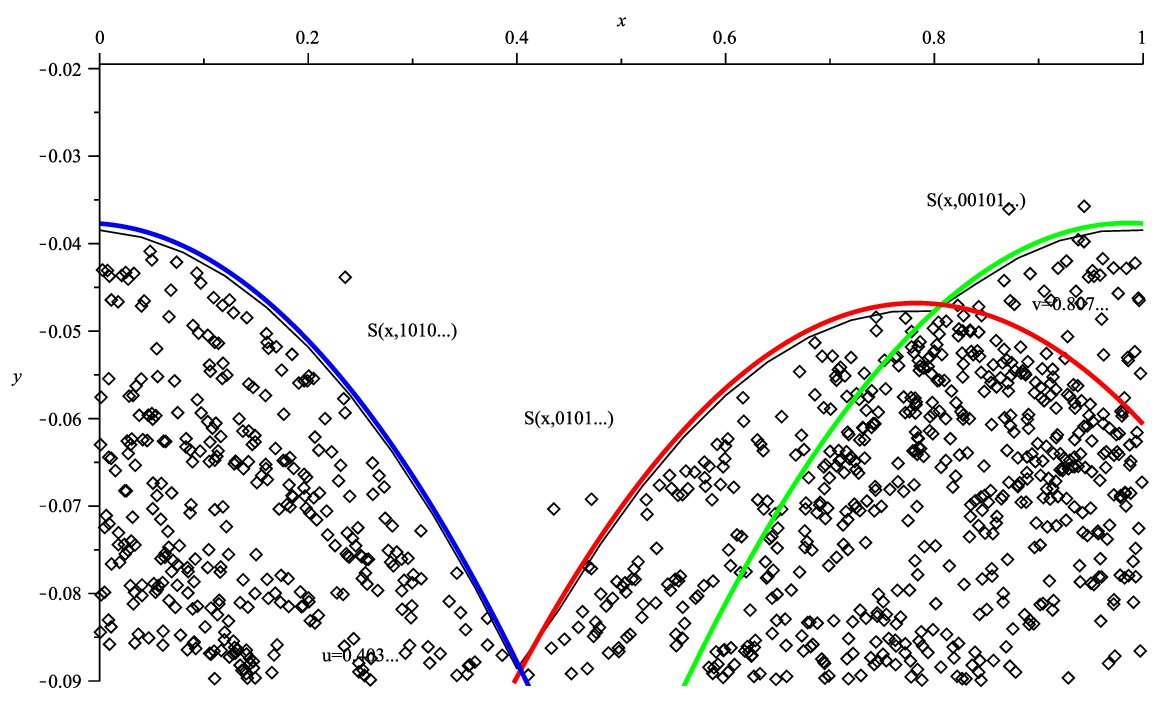}\\
  \caption{Iteration by 4000 times  of $F$.}\label{figdriftiter}
\end{figure}
\end{example}

Denote by $\mathcal{T}$ the set of turning points and $\Lambda=\cup_{n\geq 0}\,T^{n}( \mathcal{T} )$. In order to characterize the turning points we follow a discounted version of the notation introduced by \cite{Bousch0} for Sturmian measures when the symbols are $\{0,1\}$.

\begin{definition} If $ b(x) = \max_{i} \{ \lambda \, b(\tau_{i} x ) + A(\tau_{i} x)\},$ we define the remainders associated with the $b$  and $A$ as
$r(x,a)= b(x) - \lambda \, b(\tau_{a_{0}} x ) - A(\tau_{a_{0}} x),$
$$R(x)=r(x,0\cdots)-r(x,1\cdots)= (\lambda \, b(\tau_{1} x ) + A(\tau_{1} x))-(\lambda \, b(\tau_{0} x ) + A(\tau_{0} x)).$$
\end{definition}
So, $r(x,a)\geq 0$ and attains zero with the right symbol $a_{0}$. Also, $R(x)=0$, if and only if, $x$ is a turning point that is,
$\mathcal{T}=R^{-1}(0).$
\begin{definition} A continuous potential $A$ satisfy the $k$-Sturmian condition if $\sharp R^{-1}(0)= k.$ In particular, there is just $k$ turning points.
\end{definition}
In \cite{Bousch0} Sturmian measures are that ones where $k=1$. In a slightly different setting the author shows that $A(x)=cos(2\pi(x-\omega))$ satisfy the $1$-Sturmian condition for any  $\omega \in \mathbb{R}/\mathbb{Z}$.
\begin{lemma} \label{turning a(x)} $\Lambda$ contains the set of points $x$ where $a(x)$ is not a single point.
\end{lemma}
\begin{proof}
Suppose that there is two different elements $ c, d \in a(x)$, where $c$ and $d$ are of the form  $c=(i_0, i_1, ..., i_{n-1}, 1, c_{n+1}, ...)$, $d=(i_0, i_1, ..., i_{n-1}, 0, d_{n+1}, ...)$. Take $z= \tau_{i_{n-1}} \cdots \tau_{i_{0}} x $. Applying Corollary~\ref{spread} we get:
$$ c \in a(x) \Rightarrow b(x)= S(x,c) \Rightarrow b(\tau_{i_{0}} x )= S( \tau_{i_{0}} x, \sigma c) \cdots $$
$$\Rightarrow b(\tau_{i_{1}}\tau_{i_{0}} x )= S( \tau_{i_{1}}\tau_{i_{0}} x, \sigma^{2} c) \cdots \Rightarrow  b(z)= S(z,(1, c_{n+1}, ...)).$$
$$ d \in a(x) \Rightarrow b(x)= S(x,d) \Rightarrow b(\tau_{i_{0}} x )= S( \tau_{i_{0}} x, \sigma d) \cdots \Rightarrow $$
$$\Rightarrow b(\tau_{i_{1}}\tau_{i_{0}} x )= S( \tau_{i_{1}}\tau_{i_{0}} x, \sigma^{2} d) \cdots \Rightarrow  b(z)= S(z,(0, d_{n+1}, ...)).$$
Thus, $b(z) = \lambda \, b(\tau_{0} z ) + A(\tau_{0} z)$ and $b(z) = \lambda \, b(\tau_{1} z ) + A(\tau_{1} z)$, by Lemma~\ref{invariance}, that is, $z \in \mathcal{T}$. Since $T^{n}(z)=x$ we get $ x \in \Lambda$.
\end{proof}
If the turning points are finite ($A$ satisfying $k$-Sturmian condition)  and pre-periodic points for $T$, then there exists finitely many points where the optimal symbolic changes because $\sharp\Lambda < \infty$.
\begin{corollary} \label{sisi} If $\sharp\Lambda$ is finite then the graph of $b$ is a union of a finite number of $S(x,a)$.
\end{corollary}
\begin{proof}
The claim follows from Lemma~\ref{turning a(x)}.
\end{proof}
\section{Symmetric twist potentials} \label{Sym}
In this section we exhibit some explicit examples.
\begin{theorem} \label{symmet} Let $A$ be a symmetric potential, that is,
$A(1-x)=A(x)$ for any $x \in [0,1]$. In addition we assume that $A$ is twist.
Denote by $b:[0,1] \to \mathbb{R}$ the function such that
$$b(x)=\left\{
  \begin{array}{ll}
    S(x,(10)^\infty  ), &  0 \leq x\leq 1/2;  \\
    S(x,(01)^\infty  ), &  1/2< x\leq 1.
  \end{array}
\right.
$$
Then, $b$ is a $\lambda$-calibrated subaction for $A(x)$, that is, for any $x\in [0,1]$
$$ b(x) = \max_{i=0,1} \{ \lambda \, b(\tau_{i} x ) + A(\tau_{i} x)\}.$$
\end{theorem}
\begin{proof}
As $A$ is symmetric $A(1/2 -t)=A(1/2 +t)$ for $t \in [0,1/2]$. We claim that  $S(x,(10)^\infty)= S((1-x),(01)^\infty).$ Indeed
$$\tau_{0}(1-x)=\frac{1-x}{2}=1/2 - x/2\text{ and }\tau_{1}(x)=\frac{1+x}{2}=1/2 + x/2,$$
then, $A(\tau_{0}(1-x))=A(\tau_{1}(x)).$
Analogously,
$\tau_{1}\tau_{0}(1-x)=1/4 - x/4 + 1/2$ and  $\tau_{0}\tau_{1}(x)=\frac{1+x}{2}=1/4 + x/4= 1/2 - (1/4 - x/4),$
then $A(\tau_{1}\tau_{0}(1-x))=A(\tau_{0}\tau_{1}(x)),$
and so on. Thus
$$ S(x,(10)^\infty) = A(\,\tau_{1} (x)) + \lambda\, A(\,\tau_{0 }\circ \tau_{1}(x) )+  \lambda^2\, A(\,\tau_{1 }\circ \tau_{0 }\circ \tau_{1} (x))+...=$$
$$  A(\,\tau_{0}(1-x)) + \lambda\, A(\,\tau_{1}\tau_{0}(1-x) )+  \lambda^2\, A(\,\tau_{0 }\circ \tau_{1 }\circ \tau_{0} (1-x))+...=S((1-x),(01)^\infty).$$
In particular, $S(0,(10)^\infty)= S(1,(01)^\infty)$ and
$S(1/2,(10)^\infty)= S(1/2,(01)^\infty)$, that is $b(x)$ is continuous. By the twist property $S(x,(10)^\infty)$ and $ S(x,(01)^\infty)$ are transversal in $x=1/2$, then, as can not exist two points of intersection we get
\begin{center}
$S(x,(10)^\infty)> S(x,(01)^\infty)$ if $x<1/2$;\\
$S(x,(10)^\infty)< S(x,(01)^\infty)$ if $x>1/2$.\\
\end{center}
Now we will prove that the above $b$ is a $\lambda$-calibrated subaction for $A(x)$.

We divide the argument in two cases:\\
\textbf{Case 1- $x<1/2$}\\
If $i=0$ then
$$\begin{array}{cc}
\lambda b(\tau_{i} x ) + A(\tau_{i} x)&= A(\tau_{0} x) + \lambda b(\tau_{0} x )\\
&= A(\tau_{0} x) + \lambda S(\tau_{0} x,(10)^\infty)\\
&= S(x,(01)^\infty)< S(x,(10)^\infty)\\
\end{array}$$
because $x<1/2$.\\
If $i=1$ then
$$\begin{array}{cc}
\lambda b(\tau_{i} x ) + A(\tau_{i} x)&= A(\tau_{1} x) + \lambda b(\tau_{1} x )\\
&= A(\tau_{1} x) + \lambda S(\tau_{1} x,(01)^\infty)\\
&= S(x,(10)^\infty)
\end{array}$$
because $\tau_{1}x>1/2$.\\
Thus,
$ \max_{i=0,1} \{ \lambda \, b(\tau_{i} x ) + A(\tau_{i} x)\}=S(x,(10)^\infty)=b(x)$
if $x<1/2$.\\
\textbf{Case 2- $x>1/2$}\\
If $i=0$ then
$$\begin{array}{cc}
\lambda b(\tau_{i} x ) + A(\tau_{i} x)&= A(\tau_{0} x) + \lambda b(\tau_{0} x )\\
&= A(\tau_{0} x) + \lambda S(\tau_{0} x,(10)^\infty)\\
&= S(x,(01)^\infty)
\end{array}$$
because $\tau_{0} x < 1/2$. If $i=1$ then
$$\begin{array}{cc}
\lambda b(\tau_{i} x ) + A(\tau_{i} x)&= A(\tau_{1} x) + \lambda b(\tau_{1} x )\\
&= A(\tau_{1} x) + \lambda S(\tau_{1} x,(01)^\infty)\\
&= S(x,(10)^\infty) < S(x,(01)^\infty)
\end{array}$$
because $x > 1/2$. Thus,
$ \max_{i=0,1} \{ \lambda \, b(\tau_{i} x ) + A(\tau_{i} x)\}=S(x,(01)^\infty)=b(x)$
if $x  >1/2$.
\end{proof}
It follows from above that in this case the $\lambda$-calibrated subaction is piecewise differentiable if $A$ is differentiable. It is piecewise analytic (two domains of analyticity) if $A$ is analytic.
\section{A characterization of when the boundary is piecewise smooth in the case of period $2$ and $3$ } \label{tre}
We will present a characterization of when the boundary is piecewise smooth in the case the $\lambda$-maximizing probability has period $2$ and  $3$ (see Theorem \ref{okl}).
As we know that a $\lambda$-calibrated subaction for $A(x)$, that is, for any $x\in [0,1]$
$ b(x) = \max_{i=0,1} \{ \lambda \, b(\tau_{i} x ) + A(\tau_{i} x)\},$ is unique
and $\displaystyle \bar{b}(x) = \sup_{a \in \{0,1\}^{\mathbb{N}}} S(x,a)$, is also a solution of this equation, we get that the superior envelope of
$\{S(x,a) \, | \, a\in \{0,1\}^{\mathbb{N}}\}$
is piecewise regular as much as $A$.
\begin{lemma} \label{transv1} If $0<x<y<1$ and $S(x, a')=\displaystyle\sup_{a \in \{0,1\}^{\mathbb{N}}} S(x,a)$ and $S(y, a'')=\displaystyle\sup_{a \in \{0,1\}^{\mathbb{N}}} S(y,a)$ then there is $x<z<y$ such that $S(z, a')=S(z, a'')$. In particular, from the twist property, $a'>a''$ and
\begin{center}
$S(x,a')> S(x,a'')$ if $x<z$;\\
$S(x,a')< S(x,a'')$ if $x>z$.\\
\end{center}
\end{lemma}
We assume here that $T(x)$ is the transformation $2\, x$ (mod 1).
Now we going to consider the case where the maximizing measure is supported  in a periodic orbit of period 3. We know that if the minimum period is 3 there is just two possible periodic sequences: $(100100100...)=(100)^{\infty}$ and $(110110110...)=(110)^{\infty}$.
We choose the case $(110110110...)=(110)^{\infty}$ with the correspondent periodic point $x_0= 3/7< T^{2}(x_0)=5/7 < T(x_0)=6/7$.
\begin{figure}[h!]
  \centering
  \includegraphics[scale=0.9,angle=0]{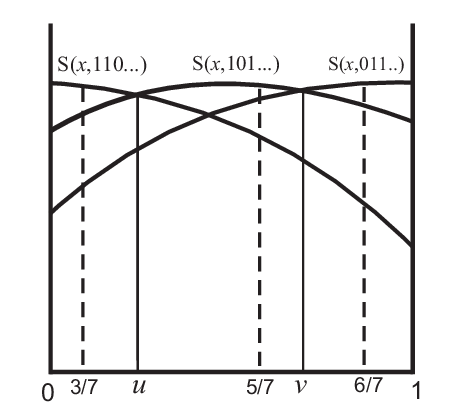}\\
  \caption{ }\label{figperiodthree}
\end{figure}
\begin{theorem} \label{okl} Let $A$ be a twist potential such that\\
$S(x_0, (110)^{\infty})=\displaystyle\sup_{a \in \{0,1\}^{\mathbb{N}}} S(x_0,a) ,S(T(x_0), (011)^{\infty})=\displaystyle\sup_{a \in \{0,1\}^{\mathbb{N}}} S(T(x_0),a)$,\\ $S( T^{2}(x_0), (101)^{\infty})=\displaystyle\sup_{a \in \{0,1\}^{\mathbb{N}}} S(T^{2}(x_0),a)$,  where $T^{3}(x_0)=x_0$. Let $u \in [x_0, T^{2}(x_0)]$ and  $v \in [T^{2}(x_0), T(x_0)]$ given by Lemma~\ref{transv1}, that is, $S( u, (110)^{\infty})= S(u, (101)^{\infty})$ and
$S( v, (101)^{\infty})= S(v, (011)^{\infty})$.
Denote by $b:[0,1] \to \mathbb{R}$ the function such that
$$b(x)=\left\{
  \begin{array}{ll}
    S(x,(110)^\infty  ), &  0 \leq x\leq u;  \\
    S(x,(101)^\infty  ), &  u\leq x\leq v; \\
    S(x,(011)^\infty  ), &  v\leq x\leq 1.
  \end{array}
\right.
$$
Then, $b$ is a $\lambda$-calibrated subaction for $A(x)$, that is, for any $x\in [0,1]$
$ b(x) = \max_{i=0,1} \{ \lambda \, b(\tau_{i} x ) + A(\tau_{i} x)\},$
if and only if, $\tau_{1}[0,u] \subseteq [u, v], \; \tau_{1}[u,v] \subseteq [v, 1] \text{  and } \tau_{0}[v,1] \subseteq [0, u].$
\end{theorem}
\begin{proof}
We must to divide in several cases.\\
\textbf{Case 1:} Consider $0 \leq x\leq u$.\\
As $\tau_{0}(x)= 1/2 x  < u$ thus
$$\lambda b(\tau_{0} x ) + A(\tau_{0} x)=\lambda S(\tau_{0} x,(110)^\infty) + A(\tau_{0} x)=$$
$$=\lambda S(\tau_{0} x,(110)^\infty) + A(\tau_{0} x)=S(x,0(110)^\infty)=$$
$$=S(x,(011)^\infty)<S(x,(110)^\infty )=b(x)$$
As $\tau_{1}(x)= 1/2 x + 1/2 > u$ we have two possibilities\\
a) If $\tau_{1}(x) \in (u, v]$ then\\
$$\lambda b(\tau_{1} x ) + A(\tau_{1} x)=\lambda S(\tau_{1} x,(101)^\infty) + A(\tau_{1} x)=$$
$$=\lambda S(\tau_{1} x,(101)^\infty) + A(\tau_{1} x)=S(x,1(101)^\infty)=$$
$$=S(x,(110)^\infty)=b(x)$$
b)If $\tau_{1}(x) \in [v, 1]$ then\\
$$\lambda b(\tau_{1} x ) + A(\tau_{1} x)=\lambda S(\tau_{1} x,(011)^\infty) + A(\tau_{1} x)=$$
$$=\lambda S(\tau_{1} x,(011)^\infty) + A(\tau_{1} x)=S(x,1(011)^\infty)=$$
$$=S(x,(101)^\infty)< S(x,(110)^\infty )= b(x).$$
Thus  $$ b(x)
\left\{
  \begin{array}{ll}
    = \max_{i=0,1} \{ \lambda \, b(\tau_{i} x ) + A(\tau_{i} x)\}, & {\rm if } \; \tau_{1}[0,u] \subseteq [u, v]  \\
    < \max_{i=0,1} \{ \lambda \, b(\tau_{i} x ) + A(\tau_{i} x)\}, &  {\rm otherwise.}
  \end{array}
\right.
$$ for $0 \leq x\leq u$.\\

\textbf{Case 2:} Consider $u \leq x \leq v$.\\
As $\tau_{0}(x)= 1/2 x  < v$ we have two possibilities\\
a) If $\tau_{0}(x) \in [0, u]$ then\\
$$\lambda b(\tau_{0} x ) + A(\tau_{0} x)=\lambda S(\tau_{0} x,(110)^\infty) + A(\tau_{0} x)=$$
$$=\lambda S(\tau_{0} x,(110)^\infty) + A(\tau_{0} x)=S(x,0(110)^\infty)=$$
$$=S(x,(011)^\infty)< S(x,(101)^\infty )=b(x)$$
b)If $\tau_{0}(x) \in [u,v]$ then\\
$$\lambda b(\tau_{0} x ) + A(\tau_{0} x)=\lambda S(\tau_{0} x,(101)^\infty) + A(\tau_{0} x)=$$
$$=\lambda S(\tau_{0} x,(101)^\infty) + A(\tau_{0} x)=S(x,0(101)^\infty)=$$
$$=S(x,0(101)^\infty)<S(x,(101)^\infty )= b(x),$$
because $S(x,0(101)^\infty)>S(x,(101)^\infty )$ contradicts the twist condition, as one can see from the  Figure~\ref{figperiodthreetwist}.
\begin{figure}[h!]
  \centering
  \includegraphics[scale=0.9,angle=0]{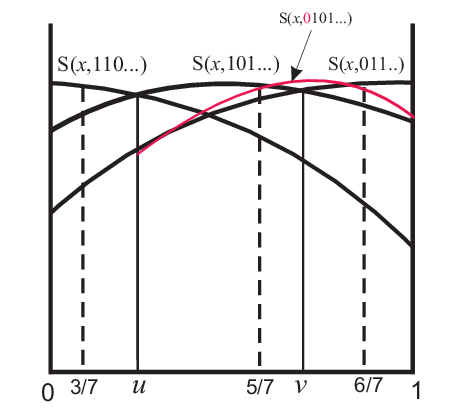}\\
  \caption{ }\label{figperiodthreetwist}
\end{figure}
As $\tau_{1}(x) \in [5/7, 6/7]$ we have two possibilities\\
a) If $\tau_{1}(x) \in [u, v]$ then\\
$$\lambda b(\tau_{1} x ) + A(\tau_{1} x)=\lambda S(\tau_{1} x,(101)^\infty) + A(\tau_{1} x)=$$
$$=\lambda S(\tau_{1} x,(101)^\infty) + A(\tau_{1} x)=S(x,1(101)^\infty)=$$
$$=S(x,(110)^\infty)< S(x,(101)^\infty )= b(x).$$

b) If $\tau_{1}(x) \in [v, 1]$ then\\
$$\lambda b(\tau_{1} x ) + A(\tau_{1} x)=\lambda S(\tau_{1} x,(011)^\infty) + A(\tau_{1} x)=$$
$$=\lambda S(\tau_{1} x,(011)^\infty) + A(\tau_{1} x)=S(x,1(011)^\infty)=$$
$$=S(x,(101)^\infty)= b(x).$$
Thus
$$ b(x)
\left\{
  \begin{array}{ll}
    = \max_{i=0,1} \{ \lambda \, b(\tau_{i} x ) + A(\tau_{i} x)\}, & {\rm if } \; \tau_{1}[u,v] \subseteq [v, 1]  \\
    < \max_{i=0,1} \{ \lambda \, b(\tau_{i} x ) + A(\tau_{i} x)\}, &  {\rm otherwise.}
  \end{array}
\right.
$$
for $u \leq x\leq v$.\\

\textbf{Case 3:} Consider $v \leq x \leq 1$.\\
As $\tau_{0}(x)= 1/2 x  < 1/2$ we have two possibilities\\
a) If $\tau_{0}(x) \in [0, u]$ then\\
$$\lambda b(\tau_{0} x ) + A(\tau_{0} x)=\lambda S(\tau_{0} x,(110)^\infty) + A(\tau_{0} x)=$$
$$=\lambda S(\tau_{0} x,(110)^\infty) + A(\tau_{0} x)=S(x,0(110)^\infty)=$$
$$=S(x,(011)^\infty)=b(x).$$
b)If $\tau_{0}(x) \in [u,v]$ then\\
$$\lambda b(\tau_{0} x ) + A(\tau_{0} x)=\lambda S(\tau_{0} x,(101)^\infty) + A(\tau_{0} x)=$$
$$=\lambda S(\tau_{0} x,(101)^\infty) + A(\tau_{0} x)=S(x,0(101)^\infty)=$$
$$=S(x,0(101)^\infty)<S(x,(011)^\infty )= b(x),$$
because $S(x,0(101)^\infty)>S(x,(011)^\infty )$ contradicts the twist (as one can see from Figure~\ref{figperiodthreetwist}).

As $\tau_{1}(x) \in [ v, 1]$ we have\\
$$\lambda b(\tau_{1} x ) + A(\tau_{1} x)=\lambda S(\tau_{1} x,(011)^\infty) + A(\tau_{1} x)=$$
$$=\lambda S(\tau_{1} x,(011)^\infty) + A(\tau_{1} x)=S(x,1(011)^\infty)=$$
$$=S(x,(101)^\infty)< S(x,(011)^\infty)= b(x).$$
Thus
$$ b(x)
\left\{
  \begin{array}{ll}
    = \max_{i=0,1} \{ \lambda \, b(\tau_{i} x ) + A(\tau_{i} x)\}, & {\rm if } \; \tau_{0}[v,1] \subseteq [0, u]  \\
    < \max_{i=0,1} \{ \lambda \, b(\tau_{i} x ) + A(\tau_{i} x)\}, &  {\rm otherwise.}
  \end{array}
\right.
$$
for $v \leq x\leq 1$.\\
\end{proof}

The characterization in the case the maximizing measure has support  in an orbit of period $n$ is similar. One needs to know the combinatorics of the position of the different points of the orbit and then proceed in an analogous  way as in the case of period $3$. We left this to the reader.

\section{Twist properties in the case $T(x)=2x\,$ ( mod 1)} \label{qua}
Let us fix $a=(a_0, a_1,  ... ) \in \{0,1\}^{\mathbb{N}}$. If $A$ is differentiable we can differentiate $S$ with respect to $x$
$$\frac{\partial S}{\partial x}(x,a) = \sum_{k=0}^{\infty} \lambda^{k} A'(\tau_{k,a} x)\frac{\partial }{\partial x} \tau_{k,a} x.$$
We observe that $\tau_{k,a} x$ has an explicit expression:
$\tau_{k,a} x = \frac{1}{2^{k+1}} x + \psi_{k}(a),$
where
$$\psi_{k}(a)= \frac{a_{0}}{2^{k+1}}+\frac{a_{1}}{2^{k}}+ ... + \frac{a_{k}}{2},$$
satisfy the recurrence relation
$2 \psi_{k+1}(a)=\psi_{k}(a) +  a_{k+1}.$
Thus
$$
\frac{\partial S}{\partial x}(x,a) = \sum_{k=0}^{\infty} \lambda^{k} A'(\tau_{k,a} x) \frac{1}{2^{k+1}}
=\frac{1}{2} \sum_{k=0}^{\infty} \left(\frac{\lambda}{2}\right)^{k} A'(\tau_{k,a} x) .
$$
Analogously,
$$
\frac{\partial^{2} S}{\partial x^{2}}(x,a) = \sum_{k=0}^{\infty} \lambda^{k} A''(\tau_{k,a} x) \frac{1}{2^{k+1}}^{2}
=\frac{1}{4} \sum_{k=0}^{\infty} \left(\frac{\lambda}{4}\right)^{k} A''(\tau_{k,a} x),
$$
in particular, if $A'' <0$ then $\frac{\partial^{2} S}{\partial x^{2}}(x,a) <0, \forall a \in \Sigma$. Even if $A$ is not $C^2$ we have the concavity of $S$ from $A$:
\begin{lemma}\label{concavity} Let $A$ be a $C^0$ potential in $S^1$. 

If $A$ is concave (strictly)  then $S(x,a) \text{ is concave (strictly)}, \forall a \in \Sigma$.
\end{lemma}
\begin{proof}
Fixed $a \in \Sigma$ consider $ x< y$ and $t \in [0,1]$ then
$$S((1-t)x + t y,a)= \sum_{k=0}^{\infty} \lambda^{k} A(\tau_{k,a} [(1-t)x + t y]).$$
Since $\tau_{k,a} (1-t)x + t y = \frac{1}{2^{k+1}} [(1-t)x + t y] + \psi_{k}(a)= (1-t)\tau_{k,a}x + t \tau_{k,a} y,$ we get
$$S((1-t)x + t y,a)= \sum_{k=0}^{\infty} \lambda^{k} A((1-t)\tau_{k,a}x + t \tau_{k,a} y)\geq$$
$$ \geq \sum_{k=0}^{\infty} \lambda^{k} [(1-t) A(\tau_{k,a}x ) + t A(\tau_{k,a} y)] = (1-t)S(x ,a)+ t S( y,a).$$
\end{proof}
\subsection{Formal  Computations}
First we prove two technical lemmas about recursive sums.
\begin{lemma}\label{quadsumderiv} Let $\psi_{k}(a)$ be the function defined above, then $\displaystyle\sum_{k=0}^{\infty} \left(\frac{\lambda}{2}\right)^{k} \psi_{k}(a)=\frac{2}{4-\lambda} Z(a),$
where $\displaystyle Z(a)= \sum_{k=0}^{\infty} \left(\frac{\lambda}{2}\right)^{k} a_{k}$.
\end{lemma}
\begin{proof}
Consider $H=\sum_{k=0}^{\infty} \left(\frac{\lambda}{2}\right)^{k} \psi_{k}(a)$ then

$$\begin{array}{ccc}
2H&=2\sum_{k=0}^{\infty} \left(\frac{\lambda}{2}\right)^{k} \psi_{k}(a)
&=2\psi_{0}(a) +\sum_{k=1}^{\infty} \left(\frac{\lambda}{2}\right)^{k} 2\psi_{k}(a)\\
&=a_{0} + \sum_{k=1}^{\infty} \left(\frac{\lambda}{2}\right)^{k} [\psi_{k-1}(a) + a_{k}]
&=a_{0} + \sum_{k=1}^{\infty} \left(\frac{\lambda}{2}\right)^{k} \psi_{k-1}(a)+ \sum_{k=1}^{\infty} \left(\frac{\lambda}{2}\right)^{k} a_{k}\\
&= \frac{\lambda}{2} \sum_{k=1}^{\infty} \left(\frac{\lambda}{2}\right)^{k-1} \psi_{k-1}(a)+ \sum_{k=0}^{\infty} \left(\frac{\lambda}{2}\right)^{k} a_{k}
&= \frac{\lambda}{2} \sum_{k=0}^{\infty} \left(\frac{\lambda}{2}\right)^{k} \psi_{k}(a)+ \sum_{k=0}^{\infty} \left(\frac{\lambda}{2}\right)^{k} a_{k}\\
&=\frac{\lambda}{2} H + Z(a).
\end{array}$$
Thus,
$H= \frac{2}{4-\lambda} Z(a),$
where $\displaystyle Z(a)= \sum_{k=0}^{\infty} \left(\frac{\lambda}{2}\right)^{k} a_{k}$ is the expansion in the bases $\frac{2}{\lambda}$ of the number $Z(a)$.
\end{proof}
\begin{lemma}\label{incresing repres} If $\lambda <1$ the function $Z: \Sigma \to [0,1]$  given by $\displaystyle Z(a)= \sum_{k=0}^{\infty} \left(\frac{\lambda}{2}\right)^{k} a_{k}$ is strictly increasing with respect to the lexicographical order. In particular, if $b>a$ then
$$ Z(b)-Z(a) \geq  \left(\frac{\lambda}{2}\right)^{n} \left( \frac{1-\lambda}{1- \frac{\lambda}{2}}\right),$$
where $n$ is the first digit where $a$ is different from $b$.
\end{lemma}
\begin{proof}
Take $a=(i_{0}, ..., i_{n-1}, 0, a_{n+1}, i_{n+2}, ....) < b=(i_{0}, ..., i_{n-1}, 1, b_{n+1}, b_{n+2}, ....),$ then,
$$Z(b)-Z(a)= \left(\frac{\lambda}{2}\right)^{n} (1-0) + \sum_{k=n+1}^{\infty} \left(\frac{\lambda}{2}\right)^{k} (b_{k}- a_{k})\geq$$
$$ \geq \left(\frac{\lambda}{2}\right)^{n} - \sum_{k=n+1}^{\infty} \left(\frac{\lambda}{2}\right)^{k} = \left(\frac{\lambda}{2}\right)^{n} - \frac{(\frac{\lambda}{2})^{n+1}}{1- \frac{\lambda}{2}}= $$
$$ \left(\frac{\lambda}{2}\right)^{n}\left( 1 - \frac{\frac{\lambda}{2}}{1- \frac{\lambda}{2}} \right)= \left(\frac{\lambda}{2}\right)^{n} \left( \frac{1-\lambda}{1- \frac{\lambda}{2}}\right) >0$$
\end{proof}
We are going now to compute $\frac{\partial S}{\partial x}(x,a)$ for $x^{m}$ for $m =0,1,2$.\footnote{This potentials are actually defined in $\mathbb{R}$ because they are not continuous functions on $S^1$, but some combination of $1,x, x^2, ...$ allow us to build an 1-periodic function}\\
a) $A(x)= 1$\\
In that case, $S_{1}(x,a) = \sum_{k=0}^{\infty} \lambda^{k} 1= \frac{1}{1-\lambda},$
so $\frac{\partial S}{\partial x}(x,a)=0$.\\
b)  $A(x)= x$\\
In that case,

$S_{x}(x,a) = \sum_{k=0}^{\infty} \lambda^{k} \tau_{k,a} x,$ so $\frac{\partial S}{\partial x}(x,a)= \frac{1}{2} \sum_{k=0}^{\infty} \left(\frac{\lambda}{2}\right)^{k}= \frac{1}{2-\lambda}$.\\
c) $A(x)= x^2$\\
In that case,

$S_{x^2}(x,a) = \sum_{k=0}^{\infty} \lambda^{k} (\tau_{k,a} x)^{2},$

so $\frac{\partial S}{\partial x}(x,a)= \frac{1}{2} \sum_{k=0}^{\infty} \left(\frac{\lambda}{2}\right)^{k} 2 (\tau_{k,a} x)= \sum_{k=0}^{\infty} \left(\frac{\lambda}{2}\right)^{k} (\tau_{k,a} x)$.
Thus,
$$\begin{array}{cc}
\frac{\partial S}{\partial x}(x,a) &= \sum_{k=0}^{\infty} \left(\frac{\lambda}{2}\right)^{k}  \left[\frac{1}{2^{k+1}} x + \psi_{k}(a) \right]\\
&=\frac{x}{2} \sum_{k=0}^{\infty} \left(\frac{\lambda}{4}\right)^{k} + \sum_{k=0}^{\infty} \left(\frac{\lambda}{2}\right)^{k} \psi_{k}(a)\\
&=\frac{2}{4-\lambda} x  + \sum_{k=0}^{\infty} \left(\frac{\lambda}{2}\right)^{k} \psi_{k}(a)\\
\end{array}$$
Applying Lemma~\ref{quadsumderiv} we have
$\displaystyle \begin{array}{cc}
\frac{\partial S_{x^2}}{\partial x}(x,a) &=  \frac{2}{4-\lambda} x +\frac{2}{4-\lambda} Z(a).\\
\end{array}$
\begin{theorem}\label{angle S} If $A(x)= c_{0} + c_{1}x + c_{2}x^{2}$ is 1-periodic differentiable in $S^1 -\{0\}$ then $$\frac{\partial S}{\partial x}(x,a)= \left(  \frac{c_{1}}{2-\lambda} +  \frac{2 c_{2} }{4-\lambda} x \right) +\frac{2 c_{2}}{4-\lambda} Z(a).$$
Moreover, $A$ is twist if and only if $c_2 <0$.
\end{theorem}
\begin{proof}
Using the notation $\displaystyle S_{A}(x,a) = \sum_{k=0}^{\infty} \lambda^{k} A(\tau_{k,a} x),$ one can easily  show that $S$ depends linearly of $A$. So if we have, $A(x)= c_{0} + c_{1}x + c_{2}x^{2}$, then
$$S_{A}(x,a) = c_{0} \cdot S_{1}(x,a)  + c_{1}S_{x  }(x,a)+ c_{2}S_{x^{2} }(x,a).$$
We also can compute $\frac{\partial S}{\partial x}(x,a)$,
$$\frac{\partial S}{\partial x}(x,a)= c_{0} \cdot 0 + c_{1} \frac{1}{2-\lambda} + c_{2} \left( \frac{2}{4-\lambda} x +\frac{2}{4-\lambda} Z(a)\right),$$
or
$$\frac{\partial S}{\partial x}(x,a)= \left(  \frac{c_{1}}{2-\lambda} +  \frac{2 c_{2} }{4-\lambda} x \right) +\frac{2 c_{2}}{4-\lambda} Z(a).$$
Moreover,
$$\frac{\partial S}{\partial x}(x,a) - \frac{\partial S}{\partial x}(x,a)= \frac{2 c_{2}}{4-\lambda} (Z(a) - Z(b)).$$
Remember that, if $a>b$ then $Z(a) - Z(b) >0$ (by Lemma~\ref{incresing repres}). In this way, $\frac{\partial S}{\partial x}(x,a) - \frac{\partial S}{\partial x}(x,a)<0$, if and only, if $c_2 <0$.
\end{proof}

\medskip

Suppose  that $A$ is such that the $\lambda$-maximizing probability has period $2$ and the subaction $b_\lambda$ is the envelope
of $S(x,(0,1,0,1,...))$ and $S(x,(1,0,1,0...))$. In this case, in order to get  explicit examples of the associated $b_\lambda$,  it is quite useful to have the explicit expression for the associated pre-orbits.

a) The expression of the preimages using $(0,1,0,1...)$. These are $\frac{1}{2}\,,\frac{1}{4}\,,\frac{5}{8}\,,\frac{5}{16}\,....$
We say that $\frac{1}{2}$ is the $0$ level, \,$\frac{1}{4}$ is the $1$ level\,, $\frac{5}{8}$ is the $2$ level, and so on.
One can show that, if $m$ is even level, then
$ \frac{2^{m+2}-1}{3\,\, 2^{m+1}}.$
If $m$ is in odd level the value it is the last one divided by $2$.

b) The expression of the preimages using $(1,0,1,0,...)$. These are $\frac{1}{2}\,,\frac{3}{4}\,,\frac{3}{8}\,,\frac{11}{16}\,....$
We say $\frac{1}{2}$ is the $0$ level, \,$\frac{3}{4}$ is the $1$ level\,, $\frac{3}{8}$ is the $2$ level, and so on.
One can show that, if $m$ is in odd level, then
$\frac{2^{m+2}+1}{3\,\, 2^{m+1}}.$
If $m$ is in even level the value it is the next one multiplied  by $2$.

\subsection{A special quadratic case}  \label{humaqui}
As an example let us consider $A(x)= -(x -1/2)^2=-1/4 + x - x^2 $, then in this case $c_{1}=1$ and $c_{2}=-1$
$$\frac{\partial S}{\partial x}(x,a)= \left(  \frac{1}{2-\lambda} -  \frac{2 }{4-\lambda} x \right) -\frac{2 }{4-\lambda} Z(a),$$
in particular
$$S(x,a)= S(0,a)+ \left(  \frac{1}{2-\lambda} x -  \frac{1 }{4-\lambda} x^2 \right) -\frac{2 x }{4-\lambda} Z(a).$$
Thus,
$$\Delta(x,a,b)= S(x,a)- S(x,b)= \Delta(0,a,b) -\frac{2 x }{4-\lambda} (Z(a) -Z(b)) $$
$$\Delta'(x,a,b)= S'(x,a)- S'(x,b)=-\frac{2 }{4-\lambda} (Z(a) -Z(b)). $$
This proves that $A(x)= -(x -1/2)^2=-1/4 + x - x^2 $ is twist. Indeed, if $a>b$ then $Z(a) > Z(b)$ and so $\Delta'(x,a,b) <0$. Note that when $\lambda\to 1$ the angles remain bounded away from zero.
\subsection{Crossing points for quadratic potentials}
For the case $A(x)= -(x -1/2)^2  $ we can compute explicitly the crossing points
$ x_{ab}=x $ or equivalently $ \Delta(x,a,b)=0,$
that is,
$$x_{ab} = \frac{4-\lambda}{2} \frac{\Delta(0,a,b)}{Z(a) -Z(b)}.$$

\subsection{The explicit $\lambda$-calibrated subaction for $A(x)= -(x -1/2)^2$} \label{morep}
Remember that
$ Z(a)= \sum_{k=0}^{\infty} \left(\frac{\lambda}{2}\right)^{k} a_{k}$.
Note that $Z(\,(01)^\infty )= 0 + \frac{\lambda}{2} + 0 +  (\frac{\lambda}{2} )^3 +0 +...=\frac{\lambda}{2} \frac{4}{4- \lambda^2}.$
Moreover, $Z(\,(10)^\infty )= 1+  0 + (\frac{\lambda}{2})^2 + 0 +  (\frac{\lambda}{2} )^4 +0 +...=\frac{4}{4- \lambda^2}.$
Note that
$$S(x,(01)^\infty  )= S(0,(01)^\infty )+ \left(  \frac{1}{2-\lambda} x -  \frac{1 }{4-\lambda} x^2 \right) -\frac{2 x }{4-\lambda} Z((01)^\infty )=$$
$$S(0, (01)^\infty)+ \left(  \frac{1}{2-\lambda} x -  \frac{1 }{4-\lambda} x^2 \right) -\frac{2 x }{4-\lambda} \frac{\lambda}{2} \frac{4}{4- \lambda^2}=$$
$$S(0, (01)^\infty)+ \left(  \frac{1}{2-\lambda} x -  \frac{1 }{4-\lambda} x^2 \right) -\frac{4\,\lambda\,  x }{(4-\lambda)\,(4- \lambda^2)} =$$
$$S(0, (01)^\infty)+ \frac{(8-2\, \lambda-\lambda^2)\,  x }{(4-\lambda)\,(4- \lambda^2)}  -  \frac{1 }{4-\lambda} x^2   ,$$
and
$$S(x,(10)^\infty  )= S(0,(10)^\infty )+ \left(  \frac{1}{2-\lambda} x -  \frac{1 }{4-\lambda} x^2 \right) -\frac{2 x }{4-\lambda} Z((10)^\infty )=$$
$$S(0, (10)^\infty)+ \left(  \frac{1}{2-\lambda} x -  \frac{1 }{4-\lambda} x^2 \right) -\frac{2 x }{4-\lambda}  \frac{4}{4- \lambda^2}=$$
$$S(0, (10)^\infty)+ \left(  \frac{1}{2-\lambda} x -  \frac{1 }{4-\lambda} x^2 \right) -\frac{8\,  x }{(4-\lambda)\,(4- \lambda^2)} =$$
$$S(0, (10)^\infty)+ \frac{(2\,\lambda-\lambda^2)\,  x }{(4-\lambda)\,(4- \lambda^2)}  -  \frac{1 }{4-\lambda} x^2   .$$
The value $S(0, (10)^\infty)$ will be explicitly obtained in the next proposition.

Observe that
$S(1,(01)^\infty  ) =S(0,(01)^\infty)    + \frac{(4\,-2\,\lambda) }{(4-\lambda)\,(4- \lambda^2)},$
and
$S(1,(10)^\infty  ) =S(0,(10)^\infty)    + \frac{(2\,\lambda-4) }{(4-\lambda)\,(4- \lambda^2)}.$
By symmetry (see next proposition) we have that $S(1/2,(10)^\infty  )= S(1/2,(01)^\infty  )$, and therefore
$$ S(1/2,(01)^\infty  )= S(0,(01)^\infty  ) + \frac{6+ \lambda}{4\, (4-\lambda)\, (2+\lambda)}=$$
\begin{equation} \label{esti}S(1/2,(10)^\infty  )= S(0,(10)^\infty  ) -  \frac{2}{4\, (4-\lambda)\, (2+\lambda)}.\end{equation}
In this way
$$  S(0,(10)^\infty  ) = S(0,(01)^\infty  ) + \frac{6+ \lambda\,+2}{4\, (4-\lambda)\, (2+\lambda)}.$$
\begin{example}\label{ddd} Here we consider $\lambda=0.51$ and a periodic and continuous standard twist potential on the circle
$A(x)=-(x-0.5)^2$ that has the maximizing measure in a period two orbit, the superior envelope has two differentiable pieces since the unique turning point $u=0.5$ pre-periodic.
\begin{figure}[h!]
  \centering
  \includegraphics[scale=0.5,angle=0]{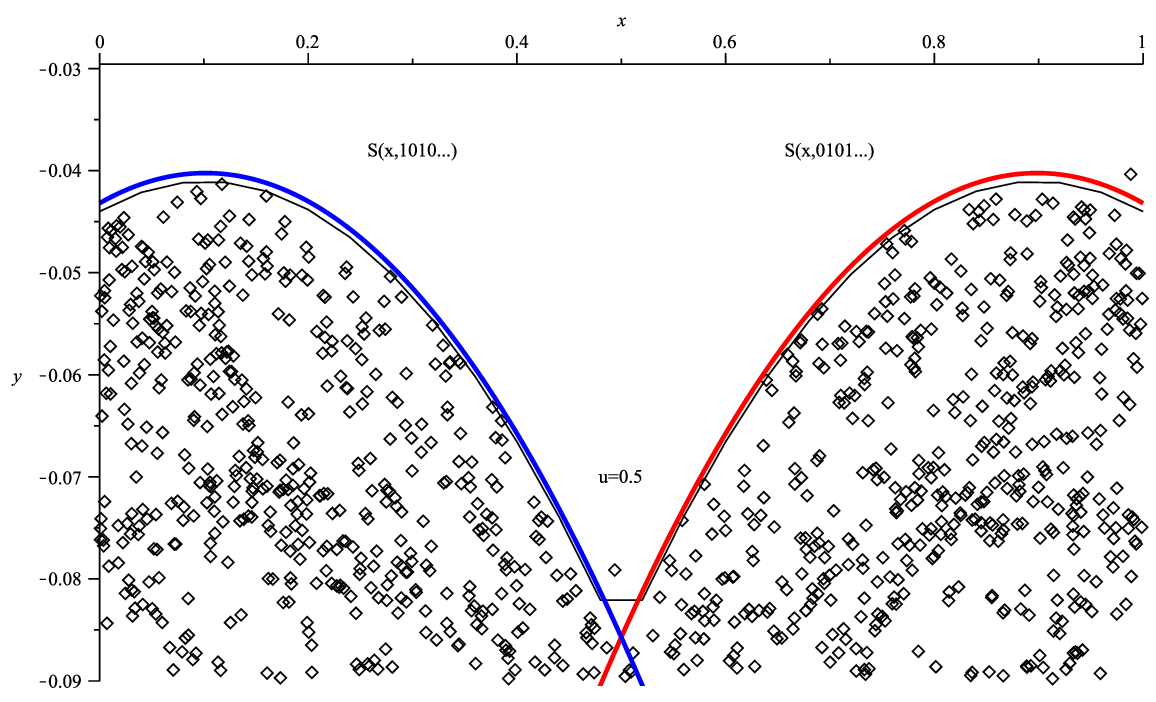}\\
  \caption{ }\label{figpotquad2piecescircle}
\end{figure}
In the figure above, the dots are the iteration of $F$, the curves are $ S(x,10101...)$ and $S(x,01010...)$. The curve dislocated is the graph of $b(x)$ computationally obtained as the superior envelope. The formal proof is given in the next.
\end{example}
\begin{proposition}
Denote by $b:S^1 \to \mathbb{R}$ the function such that for $0 \leq x\leq 1/2$, we have $b(x)=  S(x,(10)^\infty  )$, and for
$1/2\leq x\leq 1$, we have $b(x)=  S(x,(01)^\infty  )$.
Then, $b$ is a $\lambda$-calibrated subaction for $A(x)= -(x -1/2)^2$, that is, for any $x\in S^1$, $ b(x) = \max_{i} \{ \lambda \, b(\tau_{i} x ) + A(\tau_{i} x)\}= $
\begin{equation} \label{expsub} \max \{ \lambda \,b(x/2) + A(x/2),  \lambda\, b(x/2\,+\,\,1/2) + A(x/2\,+\,\,1/2)\,\} .
\end{equation}
 Moreover, $b(0) =  \frac{2\, \lambda}{4\, (4-\lambda)\, (2+\lambda) \,(\lambda-1)}$ and this provides the explicit expression of $b$.
\end{proposition}
\begin{proof}
Note that for a given $x$ we have
$$ \lambda \,b(x/2) + A(x/2)= \lambda S(x/2,(10)^\infty  ) + (-1/4 + x/2 - x^2/4)
=$$
$$ \lambda \,[\,S(0, (10)^\infty)+ \frac{(2\,\lambda-\lambda^2)\,  x }{2\, (4-\lambda)\,(4- \lambda^2)}  -  \frac{1 }{4(4-\lambda)} x^2 ] +  (-1/4 + x/2 - x^2/4)=$$
$$ \lambda S(0, (10)^\infty)+ \lambda \frac{(2\,\lambda-\lambda^2)\,  x }{2\, (4-\lambda)\,(4- \lambda^2)}  - \lambda  \frac{1 }{4(4-\lambda)} x^2 +  (-1/4 + x/2 - x^2/4)=$$
$$  \,\,S(0, (01)^\infty)\,- A(0) + \lambda \frac{(2\,\lambda-\lambda^2)\,  x }{2\, (4-\lambda)\,(4- \lambda^2)} \,\,\, - \,\,$$
$$\lambda  \frac{1 }{4(4-\lambda)} x^2 \,+  (-1/4 + x/2 - x^2/4)=$$
$$  \,\,S(0, (01)^\infty)\, + \lambda \frac{(2\,\lambda-\lambda^2)\,  x }{2\, (4-\lambda)\,(4- \lambda^2)}  - \lambda  \frac{1 }{4(4-\lambda)} x^2 \,+   x/2 - x^2/4=$$
$$  \,\,S(0, (01)^\infty)\, + \frac{(-\,\lambda^2 - 2\,\lambda+ 8)\,  x }{\, (4-\lambda)\,(4- \lambda^2)}  -   \frac{x^2 }{(4-\lambda)}  \, =S(x, (01)^\infty),$$
As $A$ is symmetric we claim that  $S(x,(10)^\infty)= S((1-x),(01)^\infty).$

Indeed,
$$ S(x,(10)^\infty) = \sum_{k=0} \lambda^k A(\,\tau_{a_k}\circ \tau_{a_{k-1}}\circ\,...\, \circ\tau_{a_0}\, ( x)  \,) =$$
$$   A(\,\tau_{1} (x)) + \lambda\, A(\,\tau_{0 }\circ \tau_{1}(x) )+  \lambda^2\, A(\,\tau_{1 }\circ \tau_{0 }\circ \tau_{1} (x))+...=$$
$$   A(\,(x+1)/2) + \lambda\, A(\,\frac{1}{2}((x+1)/2) )+  \lambda^2\, A(\,\tau_{1 } (\frac{1}{2}((x+1)/2) )))+...=$$
$$   A(\,(x+1)/2) + \lambda\, A(\,\frac{1}{2} + (x/4\,-1/4) )+  \lambda^2\, A(\,\tau_{1 } (\frac{1}{2}((x+1)/2) )))+...=$$
$$   A(\,1/2 - x) + \lambda\,  A(\,\frac{1}{2} - (x/4\,-1/4) )+  \lambda^2\, A(\,\frac{(\frac{1}{2}((x+1)/2) )+1 }{2} ))+...=$$
$$   A(\,\tau_{0} (1-x)) + \lambda\, A(\,\tau_{1 }\circ \tau_{0}(1-x) )+  \lambda^2\, A(\,\tau_{0 }\circ \tau_{1 }\circ \tau_{0} (1-x))+...= $$
\begin{equation} \label{ee7} S((1-x),(01)^\infty).
\end{equation}

Therefore, $b(x)=b(1-x)$.
Moreover,  $S(1/2,(10)^\infty)= S(1/2,(01)^\infty).$
Using this symmetry we get
$ \lambda \,b(x/2+1/2) + A(x/2+1/2)= S( x,(10)^\infty )$
from (\ref{ee7}).
From the above it follows (\ref{expsub}).
Note that from (\ref{esti}) we have
$$ b(0) = \max \{ \lambda \, b(0 ) + A(0 ),   \lambda \, b(1/2 ) + A(1/2 ) \}= $$
$$ \max \{ \lambda \, b(0 ) - \,1/4, ,\,\  \lambda \, b(1/2 )  \}=$$
$$ \max \{ \lambda \,  S(0,(10)^\infty  ) - \,1/4,  \,\, \lambda \,  S(1/2,(10)^\infty  )  \}=$$
$$ \max \{ \lambda \,  S(0,(10)^\infty  ) - \,1/4, \,\,  \lambda \,  S(0,(10)^\infty  ) -  \frac{2\, \lambda}{4\, (4-\lambda)\, (2+\lambda)}  \}=$$
$$ \max \{ \lambda \,  b(0) - \,1/4, \,\,  \lambda \,  b(0) -  \frac{2\, \lambda}{4\, (4-\lambda)\, (2+\lambda)}  \}=  \lambda \,  b(0) -  \frac{2\, \lambda}{4\, (4-\lambda)\, (2+\lambda)} .$$
In this way $ S(0,(10)^\infty  ) = b(0) =  \frac{2\, \lambda}{4\, (4-\lambda)\, (2+\lambda) \,(\lambda-1)}.$
\end{proof}
\section{Worked examples and computer simulations} \label{WE}
In the simulations we consider  the function $S: (S^1, \{1,2,..,d\}^\mathbb{N}) \to \mathbb{R}$ given by
$S(x,a) = \sum_{k=0} \lambda^k A(\,(\tau_{a_k}\circ \tau_{a_{k-1}}\circ\,...\, \circ\tau_{a_0})\, (x)  \,) ,$ and, $a=(a_0,a_1,a_2,...).$ The dynamics is defined by the inverse branches of $2x {\rm mod} 1$, that is $\tau_0=0.5 x$, $\tau_1=0.5x+0.5$ , $A(x)$ is a potential and  $\lambda=0.51$.
We will build examples where $a=(a_0,a_1,a_2,...)$ is truncated  in $a_7$, and the dots represents the iteration of typical orbits by $F(x,s) = ( T(x) , \lambda \, s + A(x)), \;(x,s)\in S^1 \times \mathbb{R}$ producing a picture of the superior envelope of the attractor.
\begin{example}\label{aaa} Here we consider a periodic and continuous potential on the circle  $A(x)=-(x-0.5)^2 + \varepsilon \psi(x) - drift$ for $\varepsilon=0.05$, $drift=0.2$ and $\psi(x)=(x-x^2) (1+3\,x+9/2\,{x}^{2}+9/2\,{x}^{3}+{\frac {27}{8}}\,{x}^{4}+{\frac {81}{40}}\,{x}^{5})$. Since,  $-(x-0.5)^2$ is twist and has the maximizing measure in a period two orbits the same is true for $A$, but in this case, the superior envelope has three differentiable pieces and turning points $u=0.21...$ and $v=0.60...$.
\begin{figure}[h!]
  \centering
  \includegraphics[scale=0.5,angle=0]{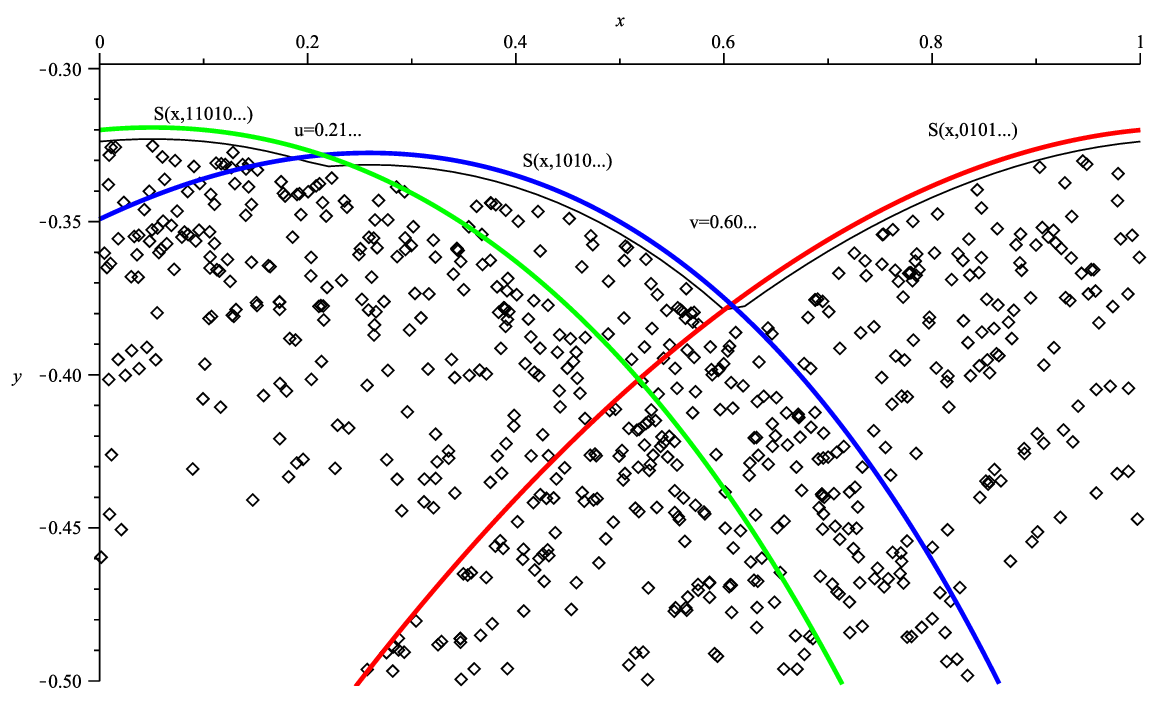}\\
  \caption{ }\label{figpotquad3piecescircle}
\end{figure}
In the figure above, the dots are the iteration of $F$, the curves are $ S(x,11010...), S(x,10101...)$ and $S(x,01010...)$. The curve dislocated is the graph of $b(x)$ computationally obtained as the superior envelope:
$$b(x):=
\left\{
  \begin{array}{lcl}
    S(x,110101....)&, & 0< x \leq u\\
    S(x,101010....)&, & u< x \leq v\\
    S(x,010101....)&, & v< x \leq 1
  \end{array}
\right..
$$
\end{example}
\begin{example}\label{bbb} Here we consider a periodic and continuous potential on the circle
$$A(x)=\left\{
  \begin{array}{ll}
    6x-3, & x<1/2 \\
    -6x+3, & x \geq 1/2
  \end{array}
\right.
$$ that is not twist but in this case, the superior envelope has two differentiable pieces since the unique turning point is pre-periodic according to Corollary~\ref{sisi}.
\begin{figure}[h!]
  \centering
  \includegraphics[scale=0.5,angle=0]{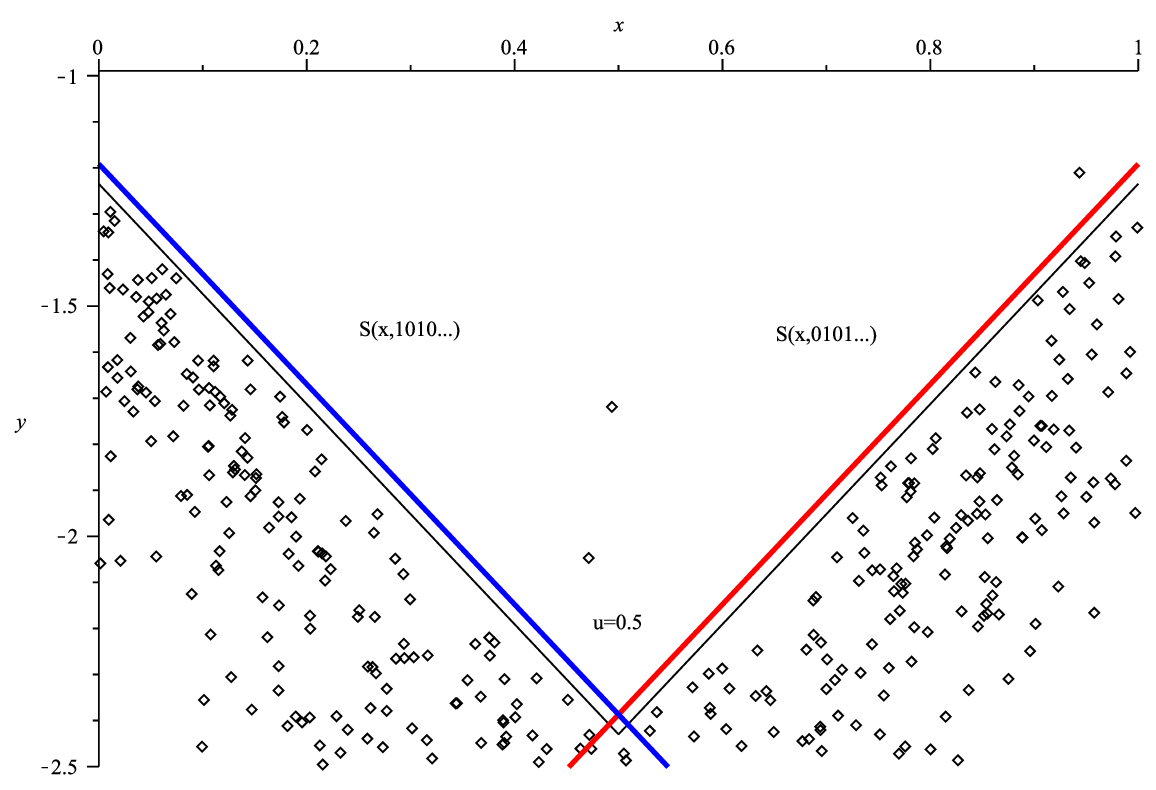}\\
  \caption{ }\label{figpottent2piecescircle}
\end{figure}
In the figure above, the dots are obtained by  the iteration of $F$ in an initial point and the curves are the graphs of $ S(x,10101...)$ and $S(x,01010...)$. The curve slightly  dislocated is the graph of $b(x)$ computationally obtained as the superior envelope.
\end{example}
\begin{example}\label{ccc} Here we consider a periodic and differentiable potential on the circle
$A(x)=-1/2-1/2\,\cos \left( 2\,\pi \,x \right)$ that is not necessarily twist but  in this case, the superior envelope has two differentiable pieces since the unique turning point is pre-periodic according to Corollary~\ref{sisi}.
\begin{figure}[h!]
  \centering
  \includegraphics[scale=0.5,angle=0]{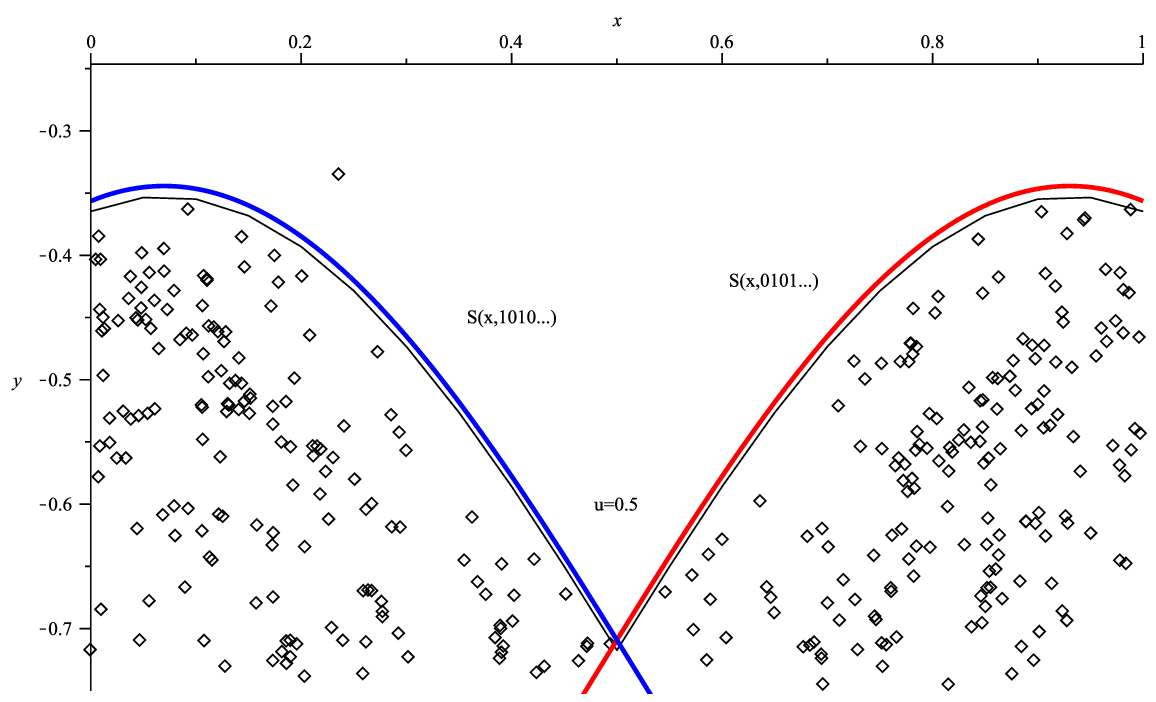}\\
  \caption{ }\label{figpotcos2piecescircle}
\end{figure}
In the figure above, the dots are the iteration of $F$ in an initial point and the curves are defined by $ S(x,10101...)$ and $S(x,01010...)$. The curve slightly  dislocated is the graph of $b(x)$ computationally obtained as the superior envelope.
\end{example}
\section{Ergodic Transport}\label{ET}
In this section $A$ is assumed to be just Lipschitz. Following the notation of  section \ref{sec2} we point out that: given $x=x_0$, there exists a sequence $x_k\in S^1$,  $k\in \mathbb{N}$, such that
$b(x_{k-1}) - \lambda b( \tau_{i_k} (x_k))-  A( \tau_{i_k} (x_k))=0.$
One can consider the probability $m_n=\sum_{j=0}^{n-1}\, \frac{1}{n} \,\delta_{\sigma^j(a)}$, where $\sigma$ is the shift, and, $a=a(x_0)$ is optimal for $x_0$.  We define  the probability $\mu_{\lambda}^*$ in $\{1,2,..,d\}^\mathbb{N}$, as any weak limit of a convergent subsequence $m_{n_k}$, $k \to \infty$ (which will be $\sigma$ invariant).
\begin{definition}
 We call
$\mu_{\lambda}^*$ a $\lambda$-dual probability for $A.$
\end{definition}
Note that from Proposition \ref{puxa} if $z$ is in the support of the $\lambda$-maximizing probability $\mu_\lambda$, then $a(z)$ can be taken as  periodic
 orbit for $\sigma$. In this case following the above reasoning we can produce a certain $\mu_{\lambda}^*$ which has support in a periodic orbit.

Consider a fixed $\overline{x}\in S^1$.
Remember that we denote
$$A^* (a) = [\lambda S(\overline{x}, \sigma(a))- S(\overline{x}, a)],$$
and in this way we get that for any $(x,a)$
$A^* (a) = A(\tau_{a_0}(x))+ [\lambda W(\tau_{a_0} (x), \sigma(a)) - W(x,a)],$
where
$W(x,a)= S(x,a) - S(\overline{x},a).$
We called such $W$ the $\lambda$-involution kernel for $A$.
We called $A^*$ is the $\lambda$-dual potential of $A$.
The main strategy is to get results for $A$ from properties of $A^*$. This is similar to the approach via  primal and dual problems  in Linear Programming.
Note that $W$ depends on the $\overline{x}$ we choose. Therefore,  $A^*=A^*_{\overline{x}}$ depends of the $\overline{x}$.
If we consider another base point $x_1$ instead $\overline{x}$, in order to get a different $W_1(x,a)= S(x,a) - S(x_1,a)$, then one can show  that the corresponding $A_1^*$ (to $A$ and $W_1$) satisfies
$ A_1^*= A^* + \lambda\,( g \circ \sigma) - g,$
for some continuous  $g$.
Note that $W- W_1$ just depends on $a$.

For the dual problem it will be necessary to consider  the following   problem: finding a function
$b^*=b_\lambda^*$ which satisfies for all $a\in \Sigma$
$$  \lambda b^*(a) = \max_{\sigma(c)=a} \{ \, b^*(c) + A^*(c)\}.$$
In fact one can do more, it is possible to  find a continuous function $b^*$  that solves  $ \lambda b^* (\sigma(c))= b^* (c) + A^* (c),\, \forall c \in \Sigma.$

Just take, as in \cite{BS},
$b^* (c) = -\sum_{j=0}^\infty \lambda^j \,A^* ( \sigma^j (c) )=$
$-\sum_{j=0}^\infty \lambda^j \, [\lambda S(\overline{x}, \sigma^{j+1} (c))- S(\overline{x}, \sigma^j (c))]= -S(\overline{x},c).$
In this case the corresponding rate function in the dual problem $R^*(c)= \lambda b^* (\sigma(c))- b^* (c) - A^* (c)$ is constant equal zero. This situation is quite different from the analogous dual problem in \cite{LOS}.
\begin{definition}
We call {\bf $b^*_\lambda$ the dual $\lambda$-calibrated subaction.}
\end{definition}
We assume, without lost of generality, that $A>0$. Then, $b>0$.
It is natural to consider the sum $\sum R^* (\sigma^n)(z)$ in the dual problem (see \cite{BLT}, \cite{LOS} and  \cite{CLO}) but now this sum is zero.
The role of the dual subactions $V$ and $V^*$ of \cite{LOS} are now played by $b$ and $b^*$, which are, respectively, the $\lambda$-calibrated subactions for $A$ and $A^*$.
Note that for all $(x,a)$
$(b^* +  b - W)(x,a) = -S(\overline{x}, a) + b(x) + S(\overline{x}, a) - S(x,a)=$
  $b(x)  - S(x,a)\geq 0.$
If $a$ is a realizer for $x$, then $(b^* + b - W)(x,a)=0$.
Given $A$ (and, a certain choice of $A^*$ and $W$) the next result  claims that the dual of $R$ is $R^*$ (which is constant equal zero), and the corresponding involution kernel is
$(b^* + b - W).$
\begin{proposition}
$$ R(\tau_{w}x)= (b^* + b - W)(x,w) - \lambda (b^* + b - W)( \tau_{w}x , \sigma(w)) .$$
\end{proposition}
\begin{proof}
We know that
$\lambda
  b^*(\sigma(w)) - b^*(w ) =   A^*(w),$ and, now using $x= T(\tau_{w}x)$, we get
$$b(x) - \lambda b(\tau_{w}x) = b(T(\tau_{w}x)) - \lambda b(\tau_{w}x) =$$
$$- A(\tau_{w}x) + A(\tau_{w}x) = R(\tau_{w}x) + A(\tau_{w}x).$$
Substituting the above in the previous equation we get
$$
(b^* + b - W)(x,w) - \lambda\, (b^* + b - W)(\tau_{w} (x) , \sigma(w))   =$$
$$
[b^*(w) - \lambda\,b^*(\sigma (w) )] + [b(x) - \lambda b(\tau_{w}x)]  - W(x,w) + \lambda W (\tau_{w}x , \sigma(w))  =$$
$$
- A^*(w) + R(\tau_{w} (x)) + A(\tau_{w}(x)) + \lambda\, W (\tau_{w} (x) , \sigma(w)) - W(x,w) =
R(\tau_{w}  (x)),$$
because $A^*(w)= A(\tau_{w}x) + \lambda\, W (\tau_{w}x , \sigma(w)) - W(x,w)$. So the claim follows.
\end{proof}
We present now a brief  outline of Transport Theory (see \cite{Vi1} \cite{Vi2} as a general reference).
\begin{definition}
We denote by $\mathcal{ K}( \mu, \mu^* )$ the set of
probabilities $\hat{\eta} (x,w)$ on $\hat{\Sigma}=S^1 \times \Sigma$, such that
$\pi_ x^*  (\hat{\eta} ) = \mu, \, \, \text{and}\,\,\pi_ w^*  (\hat{\eta} ) =
\mu^* \,.$
Each element in
$\mathcal{ K}( \mu, \mu^* )$ is called a plan.
\end{definition}
In Transport Theory one is interested in plans which minimize the integral a given lower semi-continuous  cost $c:\Sigma \to \mathbb{R}$.
The Classical Transport Theory is not a Dynamical Theory.
It is necessary to consider a dynamically defined cost in order to be able to get some results such that the optimal plan is invariant for some dynamics.
We are going to consider below the cost function $c(x,w) =- W(x,w)=-W_\lambda (x,w)$ where $W(x,w)$ is  a $\lambda$-involution  kernel of the Lipschitz potential $A$.
The Kantorovich Transport Problem: consider the minimization
problem
$$ C ( \mu, \mu^* )   \,=\,  \inf_{\hat{\eta} \in \mathcal{K}( \mu, \mu^* ) }  \int \int
- W(x,w) \, d\,\hat{\eta}. $$
\begin{definition}
A probability $  \hat{\eta} $ on $\hat{\Sigma} $ which attains such
infimum is called an optimal transport probability, or, an optimal plan,  for $c=-W$.
\end{definition}
It is natural to consider the bijective transformation  $\mathbb{T}$  which acts on $\hat{\Sigma}=S^1 \times \Sigma$ in such way that $\mathbb{T}^{-1}(x,w) = (\tau_{w}x , \sigma(w))$.
We will show  later that for $\mu_\lambda$ and $\mu_\lambda^*$ there exists a $\mathbb{T}$-invariant probability $\hat{\mu}_{min}$ which attains the optimal transport cost.
Dynamically defined costs can determine optimal plans which have dynamical properties.
\begin{definition}
A pair of continuous functions $f(x)$ and $f^{\#}(w)$ will be
called $c$-admissible (or, just admissible for short) if
$$
f^{\#} (w) =  \min_{x \in S^1} \, \{- f(x) + c(x,w) \}.
$$
\end{definition}
We denote by $\mathcal{F}$ the set of admissible pairs. The Kantorovich dual Problem: given the cost $c(x,w)$ consider  the maximization problem
$$
D (\mu  , \mu^*  )   \,=\,  \max_{(f, f^\#) \in \mathcal{ F} }\, (\, \int f  d \mu +  \int
f^\# d \mu^* \,).
$$
In this problem one is interested in any  pair (when exists)  $(f, f^\#) \in \mathcal{ F}$ which realizes the maximum in the right side of the above expression.
\begin{definition}
A pair of admissible $(f, f^\#)\in \mathcal{ F}$ which attains the
maximum value will be called an optimal Kantorovich pair.
\end{definition}
Under quite general conditions \cite{Vi1} (which are satisfied here)
$D (\mu , \mu^*  )= C (\mu  , \mu^*  ).$
We denote
$\Gamma=\Gamma_{b}=\{(x,w) \in S^1 \times \Sigma\, | \,b(x)=  (-b^*  +  W)(x,w)\}.$
A classical result in Transport Theory \cite{Vi1}:
if $\hat{\eta}$ is a
probability in  $\mathcal{ K}( \mu, \mu^* )$,  $(f,f^\#)$ is an admissible pair, and
the support of $\hat{\eta}$ is contained in the set
$\{\, (x,w)\,\, \in \hat{ \Sigma} \,| \,\mbox{such that}\,\, (f(x) + f^\#
(w))\, = \, c(x,w)\,\},$
then, $\hat{\eta}$ in an optimal plan for $c$ and $(f,f^\#)$ is an optimal pair in $\mathcal{ F}$.

This is the so called slackness condition of Linear Programming (see \cite{Vi2} Remark 5.13 page 59). This results allows one to get in some cases the solution of the primal problem
(which is looking for optimal plans) via de dual problem (which is looking for optimal pairs of functions). If you  have a good guess that a certain $\hat{\eta}$ is the optimal plan   you can try to find an admissible pair  satisfying  the above condition on the support of the plan. If you succeeded then  you show that the plan $\hat{\eta}$ is indeed the solution of the transport problem. This is the power of the dual problem  approach.

We will show that
for the problem $D (\mu_\lambda  , \mu^*_\lambda  )$ the functions
$-b$ and $-b^*$ define an optimal Kantorovich pair. From this fact becomes clear the importance of the set $\Gamma$.

Our main result in this section is:
\begin{theorem}  \label{mm} For the probabilities $\mu_\lambda  , \mu^*_\lambda $ and the cost $-W$, the associated  transport problem
is such that the functions
$-b$ and $-b^*$ define an optimal Kantorovich pair, and, the optimal plan is invariant by $\mathbb{T}$.
\end{theorem}
\begin{proof}
We claim first that
$-b$ and $-b^*$ are $-W$-admissible.
Indeed,
$p(x,w):= (b^* + b   - W)(x,w) \geq 0.$
Moreover, for each $x$ there exists a $w$ which is a realizer and then
$p(x,w)=0$.
Therefore, for each $x$ we have that
\begin{equation} \label{bb}
b(x) =  \max_{w\in \Sigma} \{-b^* (w) + W(x,w)\}= \max_{w\in \Sigma} S( x,w).
\end{equation}
\end{proof}

For each $x$ we denote $w_x\in \Sigma $ the realizer  for the above equation.
We can say that $b$ is the $W$ transform of $-b^*$ \cite{Vi1} and \cite{Vi2}.
Note that
$$\Gamma=\{(x,w) \in S^1 \times \Sigma \,| \, p(x,w)=0\}.$$
We will show that  the infimum of the cost $-W$, denoted  $c(A,\lambda)$, is equal to   $\int - b^* \,d \mu^*_\lambda + \int -b\, d \mu_\lambda.$

The next proposition  is similar to a result on \cite{LOS}. Remember that
$R=-(A - b \circ T + \lambda b)\geq 0$ is called the rate function.

\begin{proposition}{\bf
(Fundamental relation)}  For any $(x,w)$
\begin{equation}\label{FR}
 R(\tau_{w}x)= p(x,w) - \lambda\,p(\tau_{w}x , \sigma(w))
\end{equation}
Moreover,
if $\,\mathbb{T}^{-1}(x,w) = (\tau_{w}x , \sigma(w))$, then\\
a) $p -\lambda\, p \circ \mathbb{T}^{-1} (x,w) = R(\tau_{w}x) \geq 0$;\\
b) $\Gamma$ is  invariant by the action of $\mathbb{T}^{-1}$;\\
c) if $a=(i_0,i_1,i_2,...)$ is optimal for $x$, then $\sigma^n (a)$ is optimal for $( \tau_{i_{n-1}}\circ\,...\,\circ\tau_{i_1}\circ \tau_{i_0}) (x).$\\
\end{proposition}
\begin{proof}
The first claim a) is a trivial consequence of the definition of $\mathbb{T}^{-1}$. The second one it is a consequence of: $p\geq 0$, and
$$p - \lambda\, ( p \circ \mathbb{T}^{-1} )(x,w)   \geq 0  \Rightarrow p(x,w)  \geq  \lambda \, ( p \circ \mathbb{T}^{-1} )(x,w).$$
 From the above we get that in the case $(x,w)$ is optimal, then, $\mathbb{T}^{-1} (x,w) $ is also optimal. Indeed, we have that
$p(x,w)=0 \to p(\tau_{w}(x) , \sigma(w))=0.$
Item c) follows by induction.
\end{proof}
{\bf In this way $\mathbb{T}^{-n}$ spread optimal pairs}. This is a nice property that has no counterpart in the Classical Transport Theory.

Take now $(z_0,w_0)\in \Gamma_V$ and, for each $n$,
$ \hat{\mu}_n=\frac{1}{n}\, \sum_{j=0}^{n-1} \delta_{\mathbb{T}^{-j} (z_0, w_0)}.$
Note that $\mathbb{T}^{-j} (z_0, w_0)$ is optimal. The closure of the set
$\{ \mathbb{T}^{-j} (z_0, w_0),\, j\in \mathbb{N}\}$
is contained in the support of the optimal transport plan.
\begin{proposition}
We claim that any weak limit of convergent subsequence $\hat{\mu}_{n_k}$, $k \to \infty$, will define a probability $\hat{\mu}$ which is optimal for the transport problem for $-W$ and its marginals.
In this way we will show the existence of a $\mathbb{T}$-invariant probability on $S^1 \times \Sigma$ which is optimal
for the associated transport problem.
\end{proposition}
\begin{proof}
Indeed, we considered before a certain $z_0$, its realizer $w_0$, and then a convergent subsequence  $\mu_{n_k}$ (notation of last section), $n_k\to \infty,$ in order to get $\mu_\lambda$. If we consider above the corresponding subsequence  $\mathbb{T}^{-n_k} (z_0, w_0)$ we get that
the projection of $\hat{\mu}$ on the $S^1$ coordinate is $\mu_\lambda$.

In an analogous way,
we consider as before  a certain $z_0$, its realizer $w_0$, and then a convergent subsequence $m_k$ to define $\mu_\lambda^*$. If we consider above a subsequence $m_k$ of the previous sequence $n_k$ (last paragraph) we get that
the projection of $\hat{\mu}$  on the $\Sigma$ coordinate is $\mu_\lambda^*$.
As $p(x,w)= (b^* + b   - W)(x,w)$ and $p$ is zero on the orbit $\mathbb{T}^{-n_k} (z_0, w_0)$ we get that $g$ is also zero in the support of any associated weak convergent subsequence.
Then any probability $\hat{\mu}$ obtained in this way is such that  projects respectively  on  $\mu_\lambda$ and $\mu_\lambda^*$, and, moreover, satisfies
$$\int - W d \hat{\mu} = \int (- b^*) d \mu^* + \int (- b)\, d \mu_\lambda.$$
Therefore, $ C ( \mu, \mu^* )  = \int (- b^*) d \mu^* + \int (- b) d \mu_\lambda.$
\end{proof}
We point out that for  the purpose of proving the conjecture the next proposition is the key result. It is just a trivial  consequence of Theorem \ref{mm} and expression (\ref{bb}).

\begin{proposition} \label{he} Suppose  that $A$ is Lipschitz, the maximizing probability  $\mu_\lambda$ has support in a unique periodic orbit  of period $k$ and
$\mu_\lambda^*$ is a dual $\lambda$-maximizer with support on the dual periodic orbit  of period $k$ for $\sigma$, then
$$b(x) =  \max_{w\in \Sigma} \{-b^* (w) + W(x,w)\}= -b^* (a) + W(x,a)\,=$$
$\,S(x,a)= \max_{w\in \Sigma} S( x,w),$
where $a=a(x)$ is the periodic realizer of $x$. In this case $a$ is in the support of $\mu_\lambda^*$.
Moreover, the procedure: given $(z_0,a(z_0))\in \Gamma_V$ take $\hat{\nu}$
$$\lim_{n \to \infty}\, \frac{1}{n}\, \sum_{j=0}^{n-1} \delta_{\mathbb{T}^{-j} (z_0, a(z_0))}\,=\, \hat{\nu}, $$
is such that $\hat{\nu}$ is optimal and has support on a periodic orbit for $\mathbb{T}.$
In the support of $\hat{\nu}$ we have $b(x) + b^* (a(x)) = W(x,a(x))$.
\end{proposition}

\begin{proposition} \label{titi} Suppose $W$ satisfies a twist condition. Denote by $\mathfrak{w}:S^1 \to \Sigma$  the function such that for a given $x$ we have that $\mathfrak{w}(x)$ is a choice of the eventual possible $w_x$ as defined above. Then, $\mathfrak{w}$ is monotonous non-decreasing (using the lexicographic order in $\Sigma$)
\end{proposition}
The proof of this proposition is the same as the one in Proposition 6.2 in \cite{LOS} or Proposition 2.1 in \cite{CLO}.
In \cite{LM} other kinds of results in Ergodic Transport Theory are considered.

\medskip

Consider $0<\lambda<1$, and the map  $G(w,s) = ( \sigma(w) , \lambda \, s + A^* (w))$, where $G:  \{1,2,..,d\}^\mathbb{N}\times\mathbb{R} \to  \{1,2,..,d\}^\mathbb{N}\times\mathbb{R}$, and $A^*: \{1,2,..,d\}^\mathbb{N}\to \mathbb{R}$ is the dual potential.

\vspace{0.1cm}

The  dynamics of attractor for $F$  has associated to it a dual repellor naturally defined by $G$ acting on  on $\{1,2,..,d\}^\mathbb{N}\times\mathbb{R}$ . The boundary of the repellor set is the graph of $b^* : \{1,2,..,d\}^\mathbb{N} \to \mathbb{R}$ which is  the $\lambda$-dual calibrated subaction.

\end{document}